\documentclass[10pt]{article}
\usepackage{graphicx,color}
\usepackage{amsmath,amssymb,amsthm,amsfonts,amscd,amsxtra}
\usepackage[normal]{subfigure}
\usepackage{subfig}
\usepackage{verbatim}
\usepackage{xcolor}
\usepackage{bigints}
\newcommand{\ud}{\mathrm{d}}

\newtheorem{theorem}{Theorem}[section]
\newtheorem{proposition}[theorem]{Proposition}

\newtheorem{corollary}[theorem]{Corollary}
\newtheorem{lemma}[theorem]{Lemma}

\theoremstyle{definition}

\newtheorem{rem}[theorem]{Remark}

\numberwithin{equation}{section}
\newtheorem{assertion}[theorem]{Assertion}

\DeclareMathSymbol{\C}{\mathalpha}{AMSb}{"43}

\newcommand{\bsub}{\begin{subequations}}
\newcommand{\esub}{\end{subequations}$\!$}
\begin{document}
\title{Estimate for concentration level of the Adams functional and extremals
 for Adams-type inequality\thanks{Research partially supported by
FAPEG.}}

\author{
Jos\'e Francisco Alves de Oliveira
\\{\small Departamento de Matemática}
\\{\small Universidade Federal do Piau\'i}\\
 {\small  64049-550 Teresina, PI, Brazil}\\
{\small jfoliveira@ufpi.edu.br\small } \and
 Abiel Costa Macedo
\\{\small Instituto de Matem\'{a}tica e Estat\'istica}
\\{\small Universidade Federal de Goi\'as}\\
{\small 74001-970 Goi\^ania, GO, Brazil}\\
{\small abielcosta@ufg.br} }
\date{}
 \maketitle
\begin{abstract}
This paper is mainly concerned with the existence of extremals for the Adams inequality. We first establish an upper bound for the classical Adams functional along of all concentrated sequences in  the higher order Sobolev space with homogeneous Navier boundary
conditions $W^{m,\frac{n}{m}}_{\mathcal{N}}(\Omega)$, which in particular includes the classical Sobolev space $W^{m,\frac{n}{m}}_{0}(\Omega)$,  where $\Omega$ is a smooth bounded domain in Euclidean $n$-space. Secondly, based on the Concentration-compactness alternative due to Do \'{O} and Macedo, we prove the existence of extremals for the Adams inequality under Navier boundary conditions for second order derivatives at least for higher dimensions when $\Omega$ is an Euclidean ball.
\end{abstract}
\vskip 0.2truein

\noindent 2000 Mathematics Subject Classification: 35J60, 35J30, 31B30, 35B33.

\noindent Key words: Trudinger-Moser inequality, Adams inequality, Extremals, Concentrated sequences.
\section{Introduction}

Let $\Omega$ be a smooth domain in $\mathbb{R}^n$, $n\geq 2$, with  $n$-measure $|\Omega|<\infty$, and $W^{m,\frac{n}{m}}_{0}(\Omega)$ be the completion of $C^\infty_0(\Omega)$ in $W^{m,\frac{n}{m}}(\Omega)$, for positive integer  $m<n$. Given $u\in C^\infty_0(\Omega)$  we will denote
\begin{equation}\nonumber
\nabla^m u=\left\{
\begin{aligned}
&\Delta^{m/2}u, & \mbox{if} &\;\; m\;\; \mbox{is even}\\
&\nabla\Delta^{(m-1)/2}u, &\mbox{if}& \;\; m\;\;\mbox{is odd.}
\end{aligned}\right.
\end{equation}
Adams in \cite{Adams1988} proved that
\begin{equation}\label{adams}
\sup_{\underset{\left\|\nabla^m u\right\|_\frac{n}{m}\leq 1}{u\in W_0^{m,\frac{n}{m}}(\Omega),}}\int_{\Omega}e^{\beta |u|^{\frac{n}{n-m}}} \ud x < \infty,\;\; \mbox{if and only if}\;\; \beta \leq \beta_0,
\end{equation}
where
\begin{equation}\label{Adams-best-constant}
\beta_0=\beta_0(m,n)=
\begin{cases}
\frac{n}{\omega_{n-1}}\left[ \frac{\pi^{\frac{n}{2}}2^m\Gamma\left(\frac{m+1}{2}\right)}{\Gamma\left(\frac{n-m+1}{2}\right)}\right]^{{n}/{(n-m)}}, & \mbox{if}\;\; m \mbox{ is odd, } \\
\frac{n}{\omega_{n-1}}\left[ \frac{\pi^{\frac{n}{2}}2^m\Gamma\left(\frac{m}{2}\right)}{\Gamma\left(\frac{n-m}{2}\right)}\right]^{{n}/{(n-m)}}, &\mbox{if}\;\; m \mbox{ is even, }
\end{cases}
\end{equation}
in which  $\Gamma(x)=\int_0^{1} (-\ln{t})^{x-1} \, \ud t,\, x>0$
is the gamma Euler function and $\omega_{n-1}$ is the area of the surface of the unit $n$-ball.  Inequality \eqref{adams} is the extension for higher order derivatives of that classical one due to Moser \cite{Moser1970/71}, which improved the earlier results due to Trudinger \cite{Trudinger67},  Pohozaev \cite{P0H065} and Yudovich \cite{YUDO61} and it is currently known as Adams inequality or  Adams-Moser-Trudinger inequality.

Adams  inequality  has a broad range of applications in partial differential equations and geometric analysis, see for instance  \cite{doOMacedo2014,MacedodoO2015,Tarek2017,Federica2013,FardounRachid2006,FardounRachid2009,
FardounRachid2012}, and there are a lot of extensions and generalizations, among which we point out the works \cite{Alvino96,Alberico08} for  Lorentz spaces, \cite{MR3155968,Tarsi2012} for  Zygmund spaces, \cite{Fontana20} for extensions to Riesz subcritical potentials,  \cite{LiRuf2008,Ruf2005,RufSani2013,luluzhu20} for the entire  space $\Omega=\mathbb{R}^{n}$,   and \cite{MR2721666,NolascoTarantello1998,YangSuKong2016,Yang2012} for Riemannian manifolds. For more related results, see 
\cite{Cianchi08,Fontana12,Lamlu12hei, deoliveiradoomacedo2018,MR4273149,MR4289908} and references quoted therein.

Tarsi \cite{Tarsi2012} extends \eqref{adams} to functions with homogeneous Navier boundary conditions. More precisely, it was proved that
	\begin{equation}\label{Tarsi-Adams}
	\sup_{\underset{\|\nabla^m u\|_\frac{n}{m} \leq 1}{u\in W_{\mathcal{N}}^{m,\frac{n}{m}}(\Omega),}}\int_{\Omega}e^{\beta |u|^{\frac{n}{n-m}}} \ud x<\infty,\;\; \mbox{if and only if}\;\;  \beta \leq \beta_0,
	\end{equation}
where
$$
W^{m,\frac{n}{m}}_\mathcal{N} (\Omega):=\{ u \in W^{m,\frac{n}{m}}(\Omega): u_{|_{\partial \Omega}}=\Delta^j u_{|_{\partial \Omega}}=0 \mbox{ in the sense of trace}, 1\leq j< m/2 \}.
$$
We are interested in finding extremal function for the  Adams Inequality. In this direction we provide the following estimate for Adams functional along of all  concentrated sequences:

\begin{theorem}\label{thm1}
Let $m, \, n$ be positive integers, $n\geq 2$ and $n>m$, and $\Omega$ be a  smooth bounded domain in  $\mathbb{R}^n$. Let $(u_i)\subset W^{m, \frac{n}{m}}_\mathcal{N}(\Omega)$, with $\|\nabla^m u_i\|_{\frac{n}{m}}= 1$ be a sequence concentrating at $x_0\in \overline{\Omega}$, i.e.,
$$
\lim_{i\rightarrow \infty} \int_{\Omega\setminus B_r (x_0)} |\nabla^m  u_i|^{\frac{n}{m}}\ud x =0, \quad\mbox{for any}\;\; r>0.$$
Then
$$
\limsup_{i} \int_{\Omega}e^{\beta_0 |u_i|^{\frac{n}{n-m}}} \ud x\leq |\Omega|\left(1+e^{\psi\left(\frac{n}{m}\right)+\gamma}\right),
$$
where  $\gamma=\lim_{n\rightarrow\infty}\left(
\sum_{j=1}^{n}(1/j)-\ln n
 \right)$ is the Euler-Mascheroni constant and $\psi(x)=\frac{d}{dx}(\ln \Gamma(x))$ is the classical Psi-function.
\end{theorem}
\noindent Since $W^{m,\frac{n}{m}}_0 (\Omega)$ is a subspace of  $W^{m,\frac{n}{m}}_\mathcal{N} (\Omega)$, as a direct consequence of the Theorem~\ref{thm1}, we can highlight the following:
\begin{corollary}\label{corollay-main} For $\gamma$ and $\psi$  as in Theorem~\ref{thm1}, we have 
	\begin{equation}\label{levelmax}
	\limsup_{i} \int_{\Omega}e^{\beta_0 |u_i|^{\frac{n}{n-m}}} \ud x\leq |\Omega|\left(1+e^{\psi\left(\frac{n}{m}\right)+\gamma}\right),
	\end{equation}
for any $(u_i)\subset W^{m, \frac{n}{m}}_0(\Omega)$ under the same hypotheses of Theorem~\ref{thm1}.
\end{corollary}
Although Corollary \ref{corollay-main} is an easy consequence of Theorem~\ref{thm1} it is new and has merit itself. Indeed, from the  Concentration-Compactness alternative \cite[Theorem 1]{doOMacedo2014} (Theorem~\ref{lions lemma} below), in order to ensure the existence of extremal functions for the classical Adams inequality \eqref{adams},  it is now sufficient to show that  there are test functions $u\in W^{m, {n}/{m}}_0(\Omega)$ such that
\begin{equation}\nonumber
	\left\|\nabla^m u\right\|_\frac{n}{m}=1\;\;\mbox{and}\;\; \int_{\Omega}e^{\beta_0 |u|^{\frac{n}{n-m}}} \ud x> |\Omega|\left(1+e^{\psi\left(\frac{n}{m}\right)+\gamma}\right).
	\end{equation}
Hence, Corollary~\ref{corollay-main} sheds some new light on the existence of extremal functions for the inequality \eqref{adams} in the critical case $\beta=\beta_0$. Actually, for $m=1$ the existence of extremals it was proved in the series of papers \cite{Carleson-Chang,MR1171306,MR1333394,MR970849}; however concerning the higher order case $m>1$, as far as we know, there are not so many results and we can only mention Lu and Yang \cite{MR2483717}, which proved the existence of extremals in the case $m=2$ with $\Omega\subset \mathbb{R}^4$ and more recently DelaTorre and Mancini \cite{DelaTorre}, where the existence of extremals for the case $H^{m}_{0}(\Omega)$ with $\Omega\subset\mathbb{R}^{2m}$ is proved. We believe that the Corollary~\ref{corollay-main} is a significant contribution to solve  completely this question.

With this approach, using the Concentration-Compactness alternative \cite[Theorem 1]{doOMacedo2014} and Theorem~\ref{thm1},  we will state the existence of extremals for Adams inequality under homogeneous Navier boundary conditions \eqref{Tarsi-Adams} for second order derivatives, at least when  $\Omega$ is an Euclidean ball and $n$ is large enough. 
\begin{theorem}\label{thm2}
Let $B_R$ be the unit ball with radius $R>0$ centered at $0\in\mathbb{R}^n$. Then, there exists $u_0\in W^{2, \frac{n}{2}}_\mathcal{N}(B_R)$, with $\|\Delta u_0\|_\frac{n}{2}\le 1$ such that
\begin{equation}\label{max}
C_{\beta_0}(B_R)=\sup_{\underset{\|\Delta u\|_\frac{n}{2}\le 1}{u\in W_\mathcal{N}^{2,\frac{n}{2}}(B_R)}}\int_{B_R}e^{\beta_0 |u|^{\frac{n}{n-2}}} \ud x=\int_{B_R}e^{\beta_0 |u_0|^{\frac{n}{n-2}}} \ud x
\end{equation}
provided that  $n\ge 2T_0$, where $T_{0}$ is the smallest positive integer such that
$$
T_{0}\ge 1+ \frac{1+36\sigma}{17-24\gamma}+\left[1+\left(\frac{1+36\sigma}{17-24\gamma}\right)^{2}+\frac{72\sigma}{17-24\gamma}\right]^{\frac{1}{2}}\approx 51.9233
$$ 
where $\sigma=1+2/\sqrt{3}$ and $\gamma$ is the Euler-Mascheroni constant.
\end{theorem}
Theorem~\ref{thm2} is the first result on the existence of extremal function for Adams inequality under homogeneous Navier boundary conditions \cite{Tarsi2012}.

This paper is arranged as follows. In Section~\ref{section2},  we present  some notations and results  which will  be used in the next sections. Section~\ref{estimate} is devoted to prove both the  general estimate of Carleson and Chang  type (cf. \cite{Carleson-Chang}) and the Theorem~\ref{thm1} for $m=2$. In Section~\ref{proofthm2} we perform some test function computations to show that the supremum \eqref{max}  surpass the estimation given in Theorem~\ref{thm1}, which proves the Theorem~\ref{thm2}. Finally, in Section~\ref{generalcasethm1} we prove Theorem \ref{thm1}  in the general case $m\ge 2$. 

\section{Preliminaries}
\label{section2}
In this section we present both the comparison theorem of Talenti~\cite[Theorem 1]{Talenti1976} (and some generalization) and  the Concentration-Compactness alternative \cite[Theorem 1]{doOMacedo2014}. We also present some estimates for the best constant of  Hardy type inequalities due to Opic and Kufner \cite{OpicKufner1990}.
\subsection{Comparison theorem and concentration-compactness alternative}

Let $\Omega^{*}=B_R$ be the ball of radio $R>0$ centered at 0 in $\mathbb{R}^n$ such that $|\Omega^{*}|=|\Omega|$. Let $u:\Omega\rightarrow\mathbb{R}$ be a  measurable function. We denote by
$$
u^\#(s):=\inf\{t\geq 0: |\{x\in\Omega: |u(x)|>t\}|<s\}, \quad \forall s\in [0,|\Omega|],
$$
the {\it decreasing rearrangement} of $u$ and by
$$
u^*(x):=u^\#(\omega_{n}|x|^n), \quad \forall x\in \Omega^*,
$$
the {\it spherically symmetric decreasing rearrangement} of $u$, where $\omega_{n}$ is the volume of the unit ball on $\mathbb{R}^n$.
In \cite[Theorem 1]{Talenti1976} Talenti presented the result known as Talenti  comparison principle, which in particular implies in the following result
\begin{theorem}\label{talenti}
	Let $\Omega\subset \mathbb{R}^n$ be a   bounded domain and $f\in C_0^{\infty}(\Omega)$. If $u$ is a solution of
	$$
	\begin{cases}
	-\Delta u = f & \mbox{in} \quad \Omega\\
	u=0  & \mbox{in} \quad \partial\Omega,
	\end{cases}
	$$
	and $v$ is a solution of
	$$
	\begin{cases}
	-\Delta v = f^* & \mbox{in} \quad \Omega^*\\
	v=0  & \mbox{in} \quad \partial\Omega^*.
	\end{cases}
	$$
	Then $v\geq u^*$ a.e. on $\Omega^*$.
\end{theorem}
We note that the result is also true for a more general domain and function $f$, see \cite[Theorem 1]{Talenti1976}. Now by iterating the Theorem \ref{talenti}  together with the Maximum Principle we can obtain some comparison principle to the  polyharmonic equation with Navier boundary condition which can be found in \cite[Proposition 3]{GazzolaGrunauSweers2010}, that is
\begin{proposition}\label{comparizon}
Let $\Omega\subset\mathbb{R}^n$, $n\geq 2$ be a  smooth bounded domain and let $q\geq 2n/(n+2)$. Let $f\in C_0^{\infty}(\Omega)$ and let $u\in W_\mathcal{N}^{2k,q}(\Omega)$ be the unique strong solution of
\begin{equation}\label{pnbcn}
\begin{cases}
(-\Delta)^k u= f & in\,\, \Omega \\
\Delta^j u=u=0 & in\,\, \partial\Omega, \,\, j=1,2,\ldots,k-1.
\end{cases}
\end{equation}
Let $v\in W_\mathcal{N}^{2k,q}(\Omega^*)$ be the unique strong solution of
\begin{equation}\label{pnbc}
\begin{cases}
(-\Delta)^k v= f^* & in\,\, \Omega^* \\
\Delta^j v=v=0 & in\,\, \partial\Omega^*, \,\, j=1,2,\ldots,k-1.
\end{cases}
\end{equation}
\noindent Then, $v\geq u^*$ a.e on $\Omega^*$.
\end{proposition}
In \cite[Theorem 1]{doOMacedo2014}, more precisely on \cite[Remark 3]{doOMacedo2014}, is established the concentration-compactness alternative that will be used in this work to prove Theorem~\ref{thm2}.
\begin{theorem}\label{lions lemma}
Let $m$ be a  positive integer with $m<n$. Let $u_i,u\in W_{\mathcal{N}}^{m,p}(\Omega)$ and $\mu$ be a  Radon measure on $\overline{\Omega}$ such that $\|\nabla^m u_i\|_p= 1$, $u_i\rightharpoonup u$ in $W_{\mathcal{N}}^{m,p}(\Omega)$ and $|\nabla^m u_i|^p \overset{*}{\rightharpoonup} \mu$ in $\mathcal{M}(\overline{\Omega})$.
\begin{itemize}
\item[(i)] If $u\equiv0$ and  $\mu=\delta_{x_0}$, the Dirac mass concentrated at some $x_0\in\overline{\Omega}$, then, up to a subsequence,
$$
e^{\beta_0\,  \,|u_{i}|^{p/(p-1)}} \overset{*}{\rightharpoonup} c\delta_{x_0} +  \mathcal{L}_n \quad \mbox{ in } \quad \mathcal{M}(\overline{\Omega}), \quad \mbox{ for some } c \geq 0,
$$
where $\mathcal{L}_n$ is the Lebesgue measure in $\mathbb{R}^n$.
\item[(ii)] If $u\equiv0$ and $\mu$ is not a Dirac mass concentrated at one point, then there are $\gamma>1$ and $C=C(\gamma,\Omega)>0$ such that
$$
\limsup_{i}\int_{\Omega}e^{\beta_0\, \gamma \,|u_i|^{p/(p-1)}} \ \ud x \leq C.
$$
\item[(iii)] If $u\not\equiv 0$, then, for $\gamma\in [1,\eta)$,   there is a constant $C=C(\gamma,\Omega)>0$ such that
\begin{equation}\label{pipoca}
\limsup_{i}\int_{\Omega}e^{\beta_0\, \gamma \,|u_i|^{p/(p-1)}} \ \ud x\leq C,
\end{equation}
where
\begin{equation}\label{eta}
\eta=\eta_{m,n}(u):=
\begin{cases}
\left(1-\|\nabla(\Delta^k u)^*\|_p^p\right)^{-1/(p-1)} & \mbox{if} \quad m=2k+1,\\
\left(1-\|\nabla^m u\|_p^p\right)^{-1/(p-1)} & \mbox{if} \quad m=2k.
\end{cases}
\end{equation}
\end{itemize}
\end{theorem}
The above result represents an extension of  Lions concentration-compactness alternative \cite{Lions85} (see also \cite{CernyCianchiHencl13}) for higher order derivatives. 
\subsection{Hardy type inequalities}
Let $AC_{loc}(a,b)$ be the set of all locally absolutely continuous functions on interval $(a,b)$.  Denote by $W(a,b)$ the  set of all functions measurable, positive and finite almost everywhere on $(a,b)$.  Further, denote by $AC_{\mathrm{L}} (a, b)$ and $AC_{\mathrm{R}} (a, b)$ the set of all function $u\in AC_{loc}(a,b)$ such that
\begin{equation}\nonumber
\lim_{r\rightarrow a^{+}}u(r)=0
\end{equation}
and 
\begin{equation}\nonumber
\lim_{r\rightarrow b^{-}}u(r)=0,
\end{equation}
respectively.  In  \cite{OpicKufner1990},   Opic and Kufner have investigated the conditions on $p, q\in(1,\infty)$ and $v,w\in W(a,b)$ for which there is  a  constant $\mathcal{C}>0$  such that 
\begin{equation}\label{GHIneq}
	\left(\int_{a}^{b}|u(r)|^q w(r)dr\right)^\frac{1}{q}\leq \mathcal{C} \left(\int_{a}^{b}|u'(r)|^p v(r)dr\right)^\frac{1}{p},
\end{equation}
holds for every $u\in AC_{\mathrm{L}}(a,b)$ (or  $u\in AC_{\mathrm{R}}(a,b)$). In this direction,  Opic and Kufner were able to prove the following (see \cite[Theorem~1.14 and Theorem~6.2]{OpicKufner1990}):

\begin{theorem} \label{general2} Let $ 1<p\le q<\infty$ and $v,w\in W(a,b)$. Set
$$
k(q,p)=\left(1+\frac{q(p-1)}{p}\right)^\frac{1}{q}\left(1+\frac{p}{q(p-1)}\right)^{\frac{p-1}{p}}.
$$
\begin{itemize}
\item [$(i)$] Inequality \ref{GHIneq} holds for every $u\in AC_{\mathrm{L}}(a,b)$ if and only if 
\begin{equation}\nonumber
B_{\mathrm{L}}=\sup_{x\in (a,b)}\left(\int_x^b w(r) dr\right)^{\frac{1}{q}}\left(\int_{a}^{x}v^{\frac{1}{1-p}}(r) dr\right)^\frac{p-1}{p}<+\infty.
\end{equation}
Also, if $\mathcal{C}_L$ is the best possible constant in \eqref{GHIneq} we must have
$$
B_{\mathrm{L}}\leq \mathcal{C}_L\leq k(q,p)B_\mathrm{L}.
$$
\item [$(ii)$]  Inequality \ref{GHIneq} holds for every $u\in AC_{\mathrm{R}}(a,b)$ if and only if 
\begin{equation}\nonumber
B_{\mathrm{R}}=\sup_{x\in(a,b)}\left(\int_{a}^{x} w(r) dr\right)^{\frac{1}{q}}\left(\int_x^b v^{\frac{1}{1-p}}(r) dr\right)^\frac{p-1}{p}<+\infty.
\end{equation}
Also, if $\mathcal{C}_R$ is the best possible constant in \eqref{GHIneq}  we must have
$$
B_{\mathrm{R}}\leq \mathcal{C}_R\leq k(q,p)B_{\mathrm{R}}.
$$
\end{itemize}
\end{theorem}
The following  consequence of Theorem~\ref{general2} will be used in the next sections. 
\begin{corollary}\label{PCoro1}
	Let $p,q\in (1,\infty),$ $0<R<\infty $ and $ \alpha, \theta \in \mathbb{R}$. Then,     the inequality
	\begin{equation}\label{Hardy-inequality}
		\left(\int_0^R |u(r)|^q r^\theta  \ud r\right)^{\frac{1}{q}}\le \mathcal{C}\left(\int_0^R|u^{\prime}(r)|^p r^\alpha  \ud r\right)^{\frac{1}{p}}
	\end{equation}
	holds for some constant $\mathcal{C}=\mathcal{C}(p,q,\theta, \alpha, R)$ under the following conditions:
	\begin{itemize}
		\item[$(i)$] for $u\in AC_{\mathrm{L}} (0, R)$ if  the conditions
		$\alpha -p+1<0$ and $ q(\alpha-p+1)\le p(\theta+1)$  are fulfilled.  Further, if we also suppose $p=q=\alpha-\theta$ the best possible constant $\mathcal{C}_L$ must satisfy
			\begin{equation}\nonumber
		(p-1)^{\frac{p-1}{p}}\frac{1}{p-1-\alpha}\leq \mathcal{C}_L\leq  \frac{p}{p-1-\alpha}
		 \end{equation}
		\item[$(ii)$]  for $u\in AC_{\mathrm{R}} (0, R)$ if   the conditions  $\alpha -p+1>0 $ and 	$q(\alpha -p+1)\leq p(\theta +1)$ are fulfilled. Further, if we also suppose $p=q=\alpha-\theta$  the best possible constant $\mathcal{C}_{R}$ must satisfy
		\begin{equation}\nonumber
		(p-1)^{\frac{p-1}{p}}\frac{1}{\alpha-p-1}\leq \mathcal{C}_R\leq  \frac{p}{\alpha-p-1}.
		\end{equation}
	\end{itemize}
\end{corollary}
\begin{proof}
We will  apply Theorem~\ref{general2} with the special weight functions $w(r)=r^{\theta}$ and $v(r)=r^{\alpha}$, for $\alpha,\theta\in \mathbb{R}$ on the interval $(a, b)=(0,R)$. Indeed, a direct calculus shows 
\begin{equation}\nonumber
\int_{x}^{R}w(r)dr=\left\{
\begin{aligned}
& O\left(1\right)_{x\searrow0},&\;\;
\mbox{if}\;\;& \theta>-1\\
& O\left(\ln\frac{R}{x}\right)_{x\searrow0},&\;\;
\mbox{if}\;\;& \theta=-1\\
&O\left(x^{\theta+1}\right)_{x\searrow0},&\;\;
\mbox{if}\;\;& \theta<-1\\
\end{aligned}\right.
\end{equation}
and 
\begin{equation}\nonumber
\int_{0}^{x}v^{\frac{1}{1-p}}(r)dr=
\begin{aligned}
&O\left(x^{\frac{p-\alpha-1}{p-1}}\right)_{x\searrow0},&\;\;
\mbox{if}\;\;& \alpha-p+1 < 0.\\
\end{aligned}
\end{equation}
If  $\alpha -p+1<0$ it is clear that $B_{\mathrm{L}}$ is  finite for $\theta\ge -1$.  In addition, in the case $\theta<-1$, by using the assumption $(\theta+1)p\ge (\alpha-p+1)q$ we can see that $B_\mathrm{L}<+\infty$. Finally, if $\alpha -p+1<0$ and  $p=q=\alpha-\theta$ we  must have $\theta<-1$, and thus
$$
B_\mathrm{L}=(p-1)^{\frac{p-1}{p}}\frac{1}{p-1-\alpha}\;\;\;\mbox{and}\;\;\; k(p,p)=p(p-1)^{-\frac{p}{p-1}}
$$
which implies that $(i)$ holds. Now, the both conditions  $\alpha -p+1>0$ and  	$q(\alpha-p+1)\leq p(\theta +1)$ imply $\theta>-1$ and  since we have 
\begin{equation}\nonumber 
\begin{aligned}
& \int_{0}^{x}w(r)dr=O\left(x^{\theta+1}\right)_{x\searrow0},\;\; \mbox{if}\;\; \theta>-1
\end{aligned}
\end{equation}
and
\begin{equation}\nonumber
\begin{aligned}
\int_{x}^{R}v^{\frac{1}{1-p}}(r)dr=O\left( x^{-\frac{\alpha-p-1}{p-1}}\right)_{x\searrow0}\;\; \mbox{if}\;\; \alpha-p+1>0
\end{aligned}
\end{equation}
it is clear that $B_\mathrm{R}<\infty$. Also,  for $p=q=\alpha-\theta$,  it follows that
$$
B_\mathrm{R}=(p-1)^{\frac{p-1}{p}}\frac{1}{\alpha-p-1}
$$
which proves $(ii)$.
\end{proof}
\section{Adams functional along of concentrated sequences}\label{estimate}
In this section we will prove Theorem~\ref{thm1} for $m=2$. Indeed, we provide an upper bound for Adam's functional along of all concentrated sequences $(u_i)\subset W_\mathcal{N}^{2,\frac{n}{2}}(\Omega)$.
\subsection{General estimate of Carleson-Chang type}
Let $\Gamma$ be 
 the gamma Euler function
 and $\psi(x)=\frac{d}{dx}(\ln \Gamma(x))=\Gamma^{\prime}(x)/\Gamma(x)$ the classical psi-function. We recall the following properties of the psi-function which can be found in \cite[Theorem
 1.2.5]{s-functions}
 \begin{equation}\label{psi}
 \psi(1)=-\gamma,\quad
\psi(x)-\psi(1)=\displaystyle\sum_{k=0}^{\infty}\left(\frac{1}{k+1}-\frac{1}{x+k}\right),\;\; \mbox{and}\;\; \psi^{\prime}(x)=\sum_{k=0}^{\infty}\frac{1}{(x+k)^2}
 \end{equation}
 where $$\gamma=\lim_{n\rightarrow\infty}\left(
\sum_{j=1}^{n}\frac{1}{j}-\ln n\right)$$
is the Euler-Mascheroni constant.  

For $2\le p\in\mathbb{N}$ integer number, the next result was  proved  by Carleson and Chang \cite{Carleson-Chang},  and for general $2\le p\in \mathbb{R}$ it was proved by Hudson and Leckband \cite{HudsonLeck} (see also \cite{doOdeOliveira2014}).
\begin{lemma}\label{lemma2} Let $p\ge2$, $a>0$ and $\delta>0$ be real numbers. For each $w\in C[0,\infty)$, nonnegative piecewise differentiable function
satisfying $\int_{a}^{\infty}|w^{\prime}(t)|^{p}\ud t\le \delta$
we have
\begin{equation}\nonumber
\int_{a}^{\infty}e^{w^{q}(t)-t}\ud t\le
e^{w^{q}(a)-a}\frac{1}{1-\delta^{{1}/({p-1})}}\exp\left\{
\left(\frac{p-1}{p}\right)^{p-1} \frac{c^{p}\gamma_{p}}{p}+
\psi(p)+\gamma \right\},
\end{equation}
where $\gamma_{p}=\delta(1-\delta^{{1}/({p-1})})^{1-p}$,  $c=q
w^{q-1}(a)$ and $q={p}/({p-1})$.
\end{lemma}
Now, inspired by Carleson and Chang \cite{Carleson-Chang},  we will apply the above result to obtain  the following one-dimensional  estimate:
\begin{theorem} \label{realestimate}  
	Let $p\ge 2$ be a  real number. Let  $(g_i)$ be a  sequence of nonnegative continuous piecewise differentiable functions on $[0,\infty)$.  Suppose
\begin{equation}\label{concentrating}
g_i(0)=0,\quad  \int_{0}^{\infty}|g^{\prime}_i|^{p}\ud t\le 1 \quad\mbox{and}\quad \lim_{i\rightarrow\infty}\int_{0}^{A}|g^{\prime}_i|^{p}\ud t\rightarrow 0,\quad\forall\; A>0.
\end{equation}
Then
	$$
	\limsup_{i}\int_{0}^{\infty} e^{g_i^{q}(t)-t}\ud r\leq 1+e^{\psi(p)+\gamma},
	$$
where $q=p/(p-1)$ and $\psi$,  $\gamma$ are given by  \eqref{psi}.
\end{theorem}
\begin{proof} Firstly,  since $\psi(p)+\gamma>0$ (cf. \eqref{psi}), we can assume that
\begin{equation}\label{larger 2}
	\limsup_{i}\int_{0}^{\infty} e^{g_i^{q}(t)-t}\ud t>2.
\end{equation}
By H\"{o}lder inequality
$$
g_i(t)=\int_{0}^{t}g^{\prime}_i(s) ds\le \left(\int_{0}^{\infty}  |g'_i|^p \ud t\right)^{\frac{1}{p}} t^{\frac{1}{q}}\le  t^{\frac{1}{q}},\;\; \;t>0
$$
and consequently
\begin{equation}\label{gsublinear}
g_i^{q}(t)\le t,\quad \forall\; t\ge 0.
\end{equation}
For each  $i$ large enough, we claim that there exists $a_{i}$ smallest number in $[1,\infty)$ satisfying
\begin{equation}\label{rootassymp}
 g_{i}^{q}(a_{i})=a_{i}-2\ln{a_{i}} \quad\mbox{and}\quad \lim_{i\rightarrow \infty}a_i=\infty.
\end{equation} 
Indeed,  the inequality  \eqref{gsublinear} yields  $g_{i}^{q}(t)\le
t=t-2\ln^{+} t$, for all $t\in [0,1]$. Also,  if we assume that $g_{i}^{q}(t) < t-2\ln^{+} t$, for all $t \in
[1,\infty)$ it follows that
\[
\int_0^\infty e^{g_{i}^q(t)-t} \; \ud t\le 1+
\int_1^\infty e^{-2\ln t} \; \ud t =2,
\]
which contradicts   \eqref{larger 2}. Thus,  we have
$Z_i=\left\{t\in [1,\infty)\,:\, g^{q}_i(t)=t-2\ln t \right\} 
\not=\emptyset$ and we can  choose $a_i=\inf Z_i$.  Of course we have  the $a_{i}\in [1,\infty)$ and $g_{i}^{q}(a_{i})=a_{i}-2\ln{a_{i}}$.  In addition,  for each $A>0$ the assumption \eqref{concentrating} ensures
$i_{0}=i_{0}(A)$ such that
\[
\left(\int_{0}^{A}{|g_{i}^{\prime}(s)|}^{p}\ud s\right)^{1/(p-1)}
< \eta, \quad \mbox{for all} \quad i\ge i_{0}
\]
where  $\eta>0$ is chosen small enough such that $\eta t \leq
t-2\ln^{+}t,$ for all $t\in [0,\infty)$. Hence, by H\"{o}lder inequality 
$$
g_{i}^{q}(t)\le \left(\int_{0}^{A}{|g_{i}^{\prime}(s)|}^{p}\ud
s\right)^{1/(p-1)} t < \eta t \leq t-2\ln^{+}t, $$ for
all $0 \le t\le A$ and $ i \ge i_{0}$ which forces $a_{i}\ge A$,
for all $i\ge i_{0}$. Thus, we get
$a_{i}\rightarrow\infty$, as $i\rightarrow \infty$ and \eqref{rootassymp}  holds.

Now, we are using the Lemma \ref{lemma2} to complete our proof.
Indeed,  choosing  $w=g_{i}$, $a=a_{i}, \; \delta=\delta_{i}=
\int_{a_{i}}^{\infty} {|g_{i}^{\prime}}|^{p}\ud t $, we obtain
\begin{equation}\label{tail}
\int_{a_{i}}^{\infty} e^{g_{i}^{q}(t)-t}\ud t \le
\frac{1}{1-{\delta_{i}}^{1/(p-1)}} e^{K_{i}+\psi(p)+\gamma}
\end{equation}
where
$$
K_{i}= g_{i}^{q}(a_{i})\left[1+
\frac{\delta_{i}}{(p-1)(1-\delta_{i}^{1/(p-1)})^{p-1}}\right]-a_{i}.
$$
We are going to show that $$
\delta_{i}\rightarrow 0 \;\;\; \mbox{and}\;\; \;K_{i}\rightarrow 0, \;\;\;
\mbox{as}\;\;\; i\rightarrow\infty.
$$ 
We have $$g_{i}^{q}(a_{i})\le\left(\int_{0}^{a_i}|g^{\prime}_i|^{p}\ud t\right)^{\frac{q}{p}}a_i\le  (1-\delta_{i})^{1/(p-1)}a_{i},$$ 
which combined with \eqref{rootassymp} imply
$$
\delta_{i}\le 1-
\left(1-\frac{2\ln^{+}a_{i}}{a_{i}}\right)^{p-1}\le
(p-1)\frac{2\ln^{+}a_{i}}{a_{i}}\rightarrow 0,\quad\mbox{as}\quad
i\rightarrow\infty,
$$
because $1-t^{d}\le d(1-t)$, $t\ge0$ for any $d\ge 1$. We also have,
for all $i\ge i_{0}$,
$$
K_{i}\le
g_{i}^{q}(a_{i})\left[1+\frac{\delta_{i}}{p-1}+\frac{2p\delta_{i}^{q}}{p-1}\right]-a_{i}
\le 2 (p-1)^{q}a_{i}\left(\frac{2\ln^{+}a_{i}}{a_{i}}\right)^{q},
$$
which implies $K_{i}\rightarrow 0$ as $i\rightarrow \infty$. Letting $i\rightarrow\infty$ in \eqref{tail}, we obtain 
\begin{equation}\label{tail-finished}
\limsup_{i\rightarrow\infty}\int_{a_{i}}^{\infty} e^{g_{i}^{q}(t)-t}\ud t \le e^{\psi(p)+\gamma}.
\end{equation}
 It follows from \eqref{concentrating} that $g_{i}\rightarrow 0$ uniformly on compact
sets. Thus, given $\epsilon>0$ and $A>0$, we have
${g_{i}(t)}^{q}\le \epsilon$ for all $0\le t\le A$ and $i$ 
sufficiently large. Since $a_{i}$ is the smallest number such that
${g_{i}(t)}^{q}\ge t-2\ln^{+}t$, we obtain
\[
\int_{0}^{a_{i}}e^{{g^{q}_{i}(t)}-t}\ud t =
\int_{0}^{A}e^{{g^{q}_{i}(t)}-t}\ud t +
\int_{A}^{a_{i}}e^{{g^{q}_{i}(t)}-t}\ud t \nonumber \le
e^{\epsilon}\left(1- \frac{1}{e^{A}}\right)+
\left(\frac{1}{A}-\frac{1}{a_{i}}\right) \le 1 ,
\]
for $\epsilon$  small enough and  $A$ sufficiently large. On
the other hand,
$$
\int_{0}^{a_{i}}e^{{g^{q}_{i}(t)}-t}\ud t\ge
\int_{0}^{a_{i}}e^{-t}\ud t= 1-e^{-a_{i}} \rightarrow 1, \quad
\mbox{as}\quad i\rightarrow \infty.
$$
Combining the above limits, we have
$$
\lim_{i\rightarrow\infty}\int_{0}^{a_{i}}e^{{g^{q}_{i}(t)}-t}\ud t=1
$$
which together with \eqref{tail-finished} complete the proof.
\end{proof}
\subsection{Poof of the Theorem \ref{thm1}: case $m=2$}
In order to prove Theorem~\ref{thm1} we will use  Theorem~\ref{talenti} to replace the concentration sequence $(u_i)$ by another one  $ (v_i)\subset W_0^{1,n}(B_R)\cap W^{2,\frac{n}{2}}(B_R)$ which is also concentrated and satisfies $u_i^*\leq v_i$. Then, we use $ (v_i)$  to define a new sequence $(g_i)$ proper to apply the Carleson-Chang type estimate Theorem~\ref{realestimate}. Throughout this section, without loss of generality, we will assume that the functions $u_i$ belong to $C^{\infty}(\Omega)$.

\begin{lemma}\label{concentratecompa}
Let $(u_i)\subset W_\mathcal{N}^{2,\frac{n}{2}}(\Omega)$ be  a concentrated sequence such that $\|\Delta u_i\|_{\frac{n}{2}}=1$. Then there exists $(v_i)\subset W_0^{1,n}(B_R)\cap W^{2,\frac{n}{2}}(B_R)$ such that $u_i^*\leq v_i$ a.e. in $B_{R}$, $\|\Delta v_i\|_{\frac{n}{2}}=1$  and 
$$
\lim_{i\rightarrow \infty} \int_{B_R\setminus B_r} |\Delta v_i|^\frac{n}{2}\ud x=0, \;\; \mbox{for all}\;\; 0<r<R
$$
where $B_R$ is the ball centered at the origin such that $|B_R|=|\Omega|$.
\end{lemma}
\begin{proof}
Let $x_0 \in \overline{\Omega}$ such that 
$$
\lim_{i\rightarrow \infty} \int_{\Omega\setminus B_r (x_0)} |\Delta u_i|^{\frac{n}{2}}\ud x =0, \;\; \mbox{for any}\;\; r>0.
$$
For each $i$, Theorem~\ref{talenti} ensures that $v_i: B_R\rightarrow \mathbb{R}$  given by 
\begin{equation}\label{defvi}
v_i(x)=\frac{1}{n^2\omega_{n}^{\frac{2}{n}}} \int_{\omega_{n}|x|^n}^{\omega_{n}R^n} s^{\frac{2}{n}-2}\int_{0}^{s} \left(\Delta u_i\right)^\#(t) \ud t \ud s,
\end{equation}
satisfies $u_i^*\leq v_i$.
Note that
\begin{equation}\label{normalized}
\int_{B_R} |\Delta v_i|^\frac{n}{2}\ud x =1.
\end{equation}
In addition, 
\begin{equation}\label{x_00}
\int_{B_r (0)} |\Delta v_i|^\frac{n}{2}\ud x= \int_{B_r (0)} |\left(\Delta u_i\right)^*|^\frac{n}{2}\ud x \geq  \int_{B_r (x_0)\cap \Omega} |\Delta u_i|^\frac{n}{2}\ud x \rightarrow 1,
\end{equation}
as $i\rightarrow \infty$. Hence, \eqref{normalized} and \eqref{x_00} yield
\begin{equation}\label{concentrearr}
\lim_{i\rightarrow \infty} \int_{B_R\setminus B_r} |\Delta v_i|^\frac{n}{2}\ud x =1-\lim_{i\rightarrow \infty} \int_{B_r} |\Delta v_i|^\frac{n}{2}\ud x=0,
\end{equation}
for any $r\in (0,R)$.
\end{proof}
At this point, it is convenient to compare the Adams  functional along of both sequences $(u_i)$ and $(v_i)$ which were described in Lemma \ref{concentratecompa}.
Let us first denote
\begin{equation}\label{def wi}
v_i(x)=w_i(r),\;\; r=|x|.
\end{equation}
Thus,
\begin{equation}\label{wi func comp est}
\int_{\Omega}e^{\beta_0 |u_i|^{\frac{n}{n-2}}} \ud x=\int_{B_R}e^{\beta_0 (u^*_i)^{\frac{n}{n-2}}} \ud x \leq \int_{B_R}e^{\beta_0 v_i^{\frac{n}{n-2}}} \ud x=\omega_{n-1} \int_{0}^{R} e^{\beta_0 w_i^{\frac{n}{n-2}}} r^{n-1}\ud r.
\end{equation}
In addition,  we have $r^{n-1}w_i'(r)\in AC_{\mathrm{L}}(0,R)$.  Further, setting
\begin{equation}\label{Choicep=q}
p=q=\frac{n}{2}, \quad \alpha= \frac{n}{2}-\frac{n^2}{2}+n-1\quad \mbox{and}\quad  \theta=\frac{n}{2}-\frac{n^2}{2} +\frac{n}{2}-1 
\end{equation}
we have  $\alpha-p+1=n(1-n/2)<0$, $n>2$ and $q=p(\theta+1)/(\alpha-p-1)$.
Hence, the Corollary~\ref{PCoro1}, item $(i)$ yields
\begin{equation}\label{limtacao wi}
 \int_{0}^{R}  |w'_i|^\frac{n}{2} r^{\frac{n}{2}-1}\ud r\leq \mathcal{C} \int_{0}^{R}  |r^{1-n}\left(r^{n-1} w'_i\right)'|^\frac{n}{2} r^{n-1}\ud r= \frac{\mathcal{C}}{w_{n-1}} \int_{B_R} |\Delta v_i|^\frac{n}{2}\ud x= \frac{\mathcal{C}}{w_{n-1}}.
\end{equation}

In order to give a sharp estimate for the Adams functional along of the sequence $(w_i)$, we are going to estimate the constant $\mathcal{C}$ in \eqref{limtacao wi}.

\begin{lemma}\label{bestconstlemma} For any $ u \in AC_{\mathrm{L}}(0,R)$, we have
	$$
	\int_{0}^{R}  |u|^\frac{n}{2} r^{n\left( 1- \frac{n}{2}\right)-1}\ud r\leq \frac{1}{(n-2)^{\frac{n}{2}}} \int_{0}^{R}  |u'|^\frac{n}{2} r^{n\left( 1- \frac{n}{2}\right)+\frac{n}{2}-1}\ud r.
	$$
In particular, for  $w_i$ defined as in \eqref{def wi} 
\begin{equation}\label{bestconstant}
 (n-2)^{\frac{n}{2}} \omega_{n-1}\int_{0}^{R}  |w'_i|^\frac{n}{2} r^{\frac{n}{2}-1}\ud r\leq  \omega_{n-1}\int_{0}^{R}  |r^{1-n}\left(r^{n-1} w'_i\right)'|^\frac{n}{2} r^{n-1}\ud r=\int_{B_R} |\Delta v_i|^\frac{n}{2}\ud x =1
\end{equation}
holds.
\end{lemma}
\begin{proof}
It is a direct consequence of Corollary~\ref{PCoro1} with the choice  \eqref{Choicep=q}. In fact,  that choice also satisfy $p=\alpha-\theta$ and $p/(p-1-\alpha)=1/(n-2)$, then the best positive possible constant $\mathcal{C}_{L}$ must satisfy
\begin{equation}\nonumber
	(p-1)^{1-\frac{1}{p}}\frac{1}{p-1-\alpha}\le \mathcal{C}_{L}\le \frac{p}{p-1-\alpha}=\frac{1}{n-2}
	\end{equation}	
which completes the  proof.
\end{proof}

Next, we prove that the sequence $(w_i)$ defined in \eqref{def wi} is concentrated at the origin.
\begin{lemma}\label{concentrationgi} Let  $(w_i)$ be the sequence defined in \eqref{def wi}. Then, for any $r\in(0,R)$
$$
 \int_{r}^{R}  |w'_i|^\frac{n}{2} t^{\frac{n}{2}-1}\ud t\rightarrow 0,\quad \mbox{as}\;\; i\rightarrow \infty.
$$
\end{lemma}
\begin{proof}
	Note that 
$$
w_i(r)=\frac{1}{n^2\omega_{n}^{\frac{2}{n}}} \int_{\omega_{n}r^n}^{\omega_{n}R^n} s^{\frac{2}{n}-2}\int_{0}^{s} \left(\Delta u_i\right)^\#(t) \ud t \ud s,
$$
and so
$$
w_i'(s) = -\frac{1}{n\omega_{n}} s^{1-n}\int_{0}^{\omega_{n}s^n} \left(\Delta u_i\right)^\#(t) \ud t=-s^{1-n}\int_{0}^{s} \left(\Delta u_i\right)^\#(\omega_{n-1}t^{n}) t^{n-1} \ud t, \quad \mbox{for all} \quad s\in(0,R).
$$	
Hence,   the H\"{o}lder inequality and \eqref{normalized} yield
\begin{equation}\label{ascoli-bound}
|s^{n-1}w_i'(s)|\le  \left(\frac{R^{n}}{n}\right)^{\frac{n-2}{n}}\left(\int_{B_R} |\Delta v_i|^\frac{n}{2}\ud x\right)^{\frac{2}{n}}\le\left(\frac{R^{n}}{n}\right)^{\frac{n-2}{n}}, \;\; \forall\;  i\in \mathbb{N}.
\end{equation}
In addition,  for $0<r<R$  and $s,t\in [r,R]$, with $t<s$, we obtain
\begin{equation}\label{ascoli-equi}
\begin{aligned}
|s^{n-1} w_i'(s)-t^{n-1} w_i'(t)|&\leq \int_{t}^{s} \left(\Delta u_i\right)^\#(\omega_{n-1}\tau^{n}) \tau^{n-1} \ud \tau\\ 
&\leq \left(\frac{s^{n}-t^{n}}{n}\right)^{\frac{n-2}{n}}
 \left(\int_{B_R\setminus B_r (0)} |\Delta v_i|^\frac{n}{2}\ud x\right)^{\frac{2}{n}},
\end{aligned}
\end{equation}
for any  $i\in\mathbb{N}$. It follows from \eqref{ascoli-bound} and \eqref{ascoli-equi} that  $f_i: [r,R]\rightarrow\mathbb{R}$ such that $f_i(s)=s^{n-1}w^{\prime}_i(s)$ becomes uniformly bounded and equicontinuous sequence. Thus, up to a subsequence, we have  $f_i\rightarrow g$  uniformly on $[r, R]$. From Lemma~\ref{concentratecompa}, by setting $i\rightarrow\infty$ in \eqref{ascoli-equi} we conclude that $g$ must be a  constant  $c_w$.
	
Now we claim that $c_w=0$. Suppose $c_{w}>0$, then is possible to choose $i(r)$ large enough such that  for all $i\geq i(r)$
	$$
	s^{n-1} w_i'(s)\ge  c_w-\epsilon>0, \quad \forall s\in (r,R),
	$$ 
	for $\epsilon>0$  sufficiently small.   Thus,
	\begin{align*}
\int_{r}^{R}  |w'_i|^\frac{n}{2} s^{\frac{n}{2}-1}\ud s &\geq\int_{r}^{R}  \left(\frac{c_w-\epsilon}{s^{n-1}}\right)^\frac{n}{2} s^{\frac{n}{2}-1}\ud s\\
&=(c_w-\epsilon)^\frac{n}{2}  \int_{r}^{R}   s^{n-1-\frac{n^2}{2}}\ud s =O\left(r^{n-\frac{n^2}{2}}\right)_{r\searrow0},
\end{align*}
which contradicts \eqref{limtacao wi}. Analogously, the assumption $c_w<0$ leads a contradiction and  our claim is proved. Therefore, $w'_i$ converge uniformly to $0$ on $[r,R]$ , with $r\in (0,R)$ which completes the proof.
\end{proof}

Next we will complete the proof of Theorem~\ref{thm1}. Consider the  change of variable
 $$r= R e^{-\frac{t}{n}} \quad\mbox{ and }\quad g_i(t)=\omega_{n-1}^{\frac{2}{n}} n^{\frac{n-2}{n}} (n-2)w_i(r).$$ 
From \eqref{bestconstant}, we have
\begin{equation}
\int_{0}^{\infty}  |g'_i|^\frac{n}{2} \ud t=(n-2)^{\frac{n}{2}}\omega_{n-1} \int_{0}^{R} \left| w'_i(r)\right|^\frac{n}{2} r^{\frac{n}{2}-1}\ud r\leq 1.
\end{equation}
 Since $\omega_{n-1}={2\pi^{\frac{n}{2}}}/{\Gamma\left(\frac{n}{2}\right)}$ and  $\Gamma(x)=\Gamma(x+1)/x$, for $x>0$  we have
\begin{equation}\label{bcm=2}
\begin{aligned}
\beta_0=\beta_0(2,n)=\frac{n}{\omega_{n-1}}\left[ \frac{4\pi^{\frac{n}{2}}}{\Gamma\left(\frac{n}{2}-1\right)}\right]^{{n}/{(n-2)}}&=\frac{n}{\omega_{n-1}}\left[ \frac{2(n-2)\pi^{\frac{n}{2}}}{\Gamma\left(\frac{n}{2}\right)}\right]^{{n}/{(n-2)}}\\
&=\frac{n}{\omega_{n-1}}\left[ (n-2) \omega_{n-1}\right]^{{n}/{(n-2)}}\\
&=\left[ \omega_{n-1}^\frac{2}{n} n^{\frac{n-2}{n}} (n-2) \right]^{{n}/{(n-2)}}.
\end{aligned}
\end{equation}
Hence, we can write 
(cf.  \eqref{wi func comp est})
 \begin{equation}\label{estemationfunctional2}
\int_{\Omega}e^{\beta_0 |u_i|^{\frac{n}{n-2}}} \ud x\le  \omega_{n-1} \int_{0}^{R} e^{\beta_0 w_i^{\frac{n}{n-2}}} r^{n-1}	\ud r=|B_R|\int_{0}^{\infty}  e^{ g_i(t)^{\frac{n}{n-2}}-t} \ud t.
 \end{equation}
Taking into account \eqref{estemationfunctional2} and Lemma \ref{concentrationgi}, Theorem~\ref{realestimate} yields
\begin{equation}
\limsup_{i} \int_{\Omega}e^{\beta_0 |u_i|^{\frac{n}{n-2}}} \ud x\leq |B_R|\limsup_{i} \int_{0}^{\infty}  e^{ g_i(t)^{\frac{n}{n-2}}-t} \ud t\leq |\Omega|\left(1+e^{\psi\left(\frac{n}{2}\right)+\gamma}\right).
\end{equation}
\section{Proof of Theorem \ref{thm2}}\label{proofthm2}
According to Theorem~\ref{lions lemma} we only need to find some  test function such that the integral in \eqref{adams} surpass the upper bound for the concentration level given in Theorem~\ref{thm1}.  For $t\in [0,\infty)$, we set
\begin{equation}\label{test-functionw}
w(t)=\left\{\begin{aligned}
& \frac{n-2}{n}\left(\frac{n-2}{2}\right)^{-\frac{2}{n}}t,\;\; &\mbox{if}& \;\; 0\le t\le \frac{n}{2}\\
& \left(t-1\right)^{\frac{n-2}{n}}, \;\; &\mbox{if}& \;\; \frac{n}{2}< t\le \lambda\\
& \frac{n-2}{3}\left(\lambda-1\right)^{-\frac{2}{n}}\left(1-e^{\frac{3}{n}(\lambda-t)}\right)^{\frac{n-2}{n}}+\left(\lambda-1\right)^{\frac{n-2}{n}}, \;\; &\mbox{if}& \;\;  t\ge \lambda\\
\end{aligned}\right.
\end{equation}
where $\lambda>0$ will be chosen later. Then, we set
\begin{equation}\label{test-funtionu}
u(x)=\omega^{-\frac{2}{n}}_{n-1}n^{-\frac{n-2}{n}}(n-2)^{-1}w\left(n\ln\frac{R}{|x|}\right), \;\; 0<|x|\le R.
\end{equation}
Hence, we obtain $u\in W_\mathcal{N}^{2,\frac{n}{2}}(B_R)$, where $B_R$ is the ball with radius $R>0$ centered at the origin. Since $u$ is a radially symmetric function, we can write
\begin{equation}
\Delta u=\omega^{-\frac{2}{n}}_{n-1}n^{\frac{n}{2}}\left(\frac{n}{n-2}w^{\prime\prime}\left(n\ln\frac{R}{r}\right)-w^{\prime}\left(n\ln\frac{R}{r}\right)\right)\frac{1}{r^2}, \;\;\; 0<r=|x|\le R.
\end{equation}
Thus,
\begin{equation}\label{norm-L2w}
\|\Delta u\|_{\frac{n}{2}}=\left(\int_{0}^{\infty}\left|L(w^{\prime\prime},w^{\prime})\right|^{\frac{n}{2}} dt\right)^{\frac{2}{n}}.
\end{equation}
where
\begin{equation}
L(w^{\prime\prime},w^{\prime})(t)=\frac{n}{n-2}w^{\prime\prime}(t)-w^{\prime}(t),\;\; t\ge 0.
\end{equation}
\begin{lemma} For $n\ge 15$, there exists $\lambda>n/2$ such that $\|\Delta u\|_{\frac{n}{2}}\le 1$. In fact, we can choose
\begin{equation}
\lambda=1+\frac{n-2}{2}e^{b-s}
\end{equation}
where $0<s<b$, with $s$ and $b$ depending  on $n$.
\end{lemma}
\begin{proof}
	Note that
	\begin{equation}
		L(w^{\prime\prime},w^{\prime})(t)=\left\{\begin{aligned}
			&\frac{n-2}{n}\left(\frac{n-2}{2}\right)^{-\frac{2}{n}} ,\;\; &\mbox{if}& \;\; 0\le t\le \frac{n}{2}\\
			&-\frac{1}{n}\left((n-2)(t-1)^{-\frac{2}{n}}+2(t-1)^{-\frac{n+2}{n}}\right) \;\; &\mbox{if}& \;\; \frac{n}{2}< t\le \lambda\\
			& \left(\frac{n+1}{n}\right)(\lambda-1)^{-\frac{2}{n}}e^{\frac{3}{n}\lambda} e^{-\frac{3}{n}t}\;\; &\mbox{if}& \;\;  t\ge \lambda
		\end{aligned}\right.
	\end{equation}
	and
	\begin{equation}\nonumber
		\begin{aligned}
			&\int_{0}^{\infty}\left|L(w^{\prime\prime},w^{\prime})\right|^{\frac{n}{2}} dt\\
			&=\int_{0}^{\infty}\left|\chi_{(0,\frac{n}{2})}L(w^{\prime\prime},w^{\prime})-\chi_{(\frac{n}{2},\lambda)}\frac{1}{n}\left((n-2)(t-1)^{-\frac{2}{n}}+2(t-1)^{-\frac{n+2}{n}}\right)+\chi_{(\lambda,\infty)}L(w^{\prime\prime},w^{\prime})\right|^{\frac{n}{2}} dt.
		\end{aligned}
	\end{equation}
	Hence, the Minkowski inequality yields
	\begin{equation}
		\begin{aligned}
			\left(\int_{0}^{\infty}\left|L(w^{\prime\prime},w^{\prime})\right|^{\frac{n}{2}} dt\right)^{\frac{2}{n}}& \le \frac{2}{n}\left(\int_{\frac{n}{2}}^{\lambda}(t-1)^{-\frac{n+2}{2}}dt \right)^{\frac{2}{n}}\\
			&+\left(\int_{0}^{\infty}\left|\chi_{(0,\frac{n}{2})}L(w^{\prime\prime},w^{\prime})-\chi_{(\frac{n}{2},\lambda)}\frac{1}{n}\left((n-2)(t-1)^{-\frac{2}{n}}\right)+\chi_{(\lambda,\infty)}L(w^{\prime\prime},w^{\prime})\right|^{\frac{n}{2}} dt\right)^{\frac{2}{n}}\\
			&=\frac{2}{n}\left(\frac{2}{n}\left(\left(\frac{2}{n-2}\right)^{\frac{n}{2}}-\left(\frac{1}{\lambda-1}\right)^{\frac{n}{2}}\right)\right)^{\frac{2}{n}}\\
			&+\left(\int_{0}^{\frac{n}{2}}\left|L(w^{\prime\prime},w^{\prime})\right|^{\frac{n}{2}}dt +\left(\frac{n-2}{n}\right)^{\frac{n}{2}}\int_{\frac{n}{2}}^{\lambda}\frac{1}{t-1} dt +\int_{\lambda}^{\infty}|L(w^{\prime\prime},w^{\prime})|^{\frac{n}{2}} dt\right)^{\frac{2}{n}}\\
			&=\frac{2}{n}\left(\frac{2}{n}\left(\left(\frac{2}{n-2}\right)^{\frac{n}{2}}-\left(\frac{1}{\lambda-1}\right)^{\frac{n}{2}}\right)\right)^{\frac{2}{n}}\\
			&+\left(\left(\frac{n-2}{n}\right)^{\frac{n}{2}-1} +\left(\frac{n-2}{n}\right)^{\frac{n}{2}}\ln\left(\frac{2(\lambda-1)}{n-2}\right)+\frac{2}{3}\left(\frac{n+1}{n}\right)^{\frac{n}{2}}\frac{1}{\lambda-1}\right)^{\frac{2}{n}}.
		\end{aligned}
	\end{equation}
From \eqref{norm-L2w}, we have
	\begin{equation}
		\begin{aligned}
			\|\Delta u\|_{\frac{n}{2}}& \le\frac{2}{n}\left(\left(\frac{2}{n-2}\right)^{\frac{n}{2}}-\left(\frac{1}{\lambda-1}\right)^{\frac{n}{2}}\right)^{\frac{2}{n}}\\
			&+\left(\left(\frac{n-2}{n}\right)^{\frac{n}{2}-1} +\left(\frac{n-2}{n}\right)^{\frac{n}{2}}\ln\left(\frac{2(\lambda-1)}{n-2}\right)+\frac{2}{3}\left(\frac{n+1}{n}\right)^{\frac{n}{2}}\frac{1}{\lambda-1}\right)^{\frac{2}{n}}.
		\end{aligned}
	\end{equation}
	Setting
	\begin{equation}\label{lambs}
		\lambda_{s}:=1+\frac{n-2}{2}e^{b-s},\;\; \mbox{with}\;\; 0<s<b
	\end{equation}
	we can write
	\begin{equation}
		\begin{aligned}
			\|\Delta u\|_{\frac{n}{2}}& \le\frac{4}{n(n-2)}\left(1-\left(\frac{1}{e^{b-s}}\right)^{\frac{n}{2}}\right)^{\frac{2}{n}}\\
			&+\left(\left(\frac{n-2}{n}\right)^{\frac{n}{2}-1} +\left(\frac{n-2}{n}\right)^{\frac{n}{2}}(b-s)+\frac{2}{3}\left(\frac{n+1}{n}\right)^{\frac{n}{2}}\frac{2}{n-2}\frac{1}{e^{b-s}}\right)^{\frac{2}{n}}.
		\end{aligned}
	\end{equation}
	Now, setting
	\begin{equation}\label{b-choice}
		b:=\left(\frac{n}{n-2}\right)^{\frac{n}{2}}-\frac{n}{n-2}
	\end{equation}
	we have 
	\begin{equation}\label{LL2-norm}
		\begin{aligned}
			\|\Delta u\|_{\frac{n}{2}}& \le\frac{4}{n(n-2)}+\left(1-s\left(\frac{n-2}{n}\right)^{\frac{n}{2}}+\frac{4}{3}\left(\frac{n+1}{n}\right)^{\frac{n}{2}}\frac{1}{n-2}\right)^{\frac{2}{n}},\;\; \mbox{for any}\;\; 0<s<b.
		\end{aligned}
	\end{equation}
	Finally, we set
	\begin{equation}\label{s-choice}
		\begin{aligned}
			s:=\left(\frac{n}{n-2}\right)^{\frac{n}{2}}\left[1+\frac{4}{3}\left(\frac{n+1}{n}\right)^{\frac{n}{2}}\frac{1}{n-2}-\left(1-\frac{4}{n(n-2)}\right)^{\frac{n}{2}}\right].
		\end{aligned}
	\end{equation}
From \eqref{LL2-norm},  the above choice for $s$ clearly ensures  $\|\Delta u\|_{\frac{n}{2}}\le 1$ provided that $0<s<b$. Bernoulli's inequality yields
$$
\left(1-\frac{4}{n(n-2)}\right)^{\frac{n}{2}}\ge 1 -\frac{2}{n-2}, \;\; n\ge 4
$$
and since $2\le (1+\frac{1}{k})^{k}<3, k\ge 2$, we can write
\begin{align*}
		\frac{s}{b}&=\frac{\left(1+\frac{2}{n-2}\right)^{\frac{n-2}{2}}\left[1+\frac{4}{3}\left(1+\frac{1}{n}\right)^{\frac{n}{2}}\frac{1}{n-2}-\left(1-\frac{4}{n(n-2)}\right)^{\frac{n}{2}}\right]}{\left(1+\frac{2}{n-2}\right)^{\frac{n-2}{2}}-1}\\
		&\leq 4\left(1+\frac{1}{n}\right)^{\frac{n}{2}}\frac{1}{n-2}+\frac{6}{n-2}\\
		&=\frac{6}{n-2}\left[\frac{2}{3}\left(1+\frac{1}{n}\right)^{\frac{n}{2}}+1\right]<\frac{6}{n-2}\left(\frac{2\sqrt{3}}{3}+1\right),
\end{align*}
where the right side is decreasing on $n$ and for $n=15$
	$$
	\frac{6}{13}\left(\frac{2\sqrt{3}}{3}+1\right)<1.
	$$
Hence, we have $s<b$ for $n\ge 15$.
\end{proof} 
\begin{rem}  Note that the Bernoulli's inequality 
 and $(1+\frac{1}{n})^{n/2}<\sqrt{3}$ allow us to obtain the estimate
\begin{equation}\label{s-upb}
\begin{aligned}
	0<s &\leq \left(\frac{n}{n-2}\right)^{\frac{n}{2}}\left[\frac{4}{3}\left(1+\frac{1}{n}\right)^{\frac{n}{2}}\frac{1}{n-2}+\frac{2}{n-2}\right]<\left(\frac{n}{n-2}\right)^{\frac{n}{2}}\left[\left(\frac{2\sqrt{3}}{3}+1\right)\frac{2}{n-2}\right],\;\; n\ge 4.
\end{aligned}
\end{equation}
\end{rem}
\begin{lemma} Let $u$ given in \eqref{test-funtionu} with $\lambda=1+\frac{n-2}{2}e^{b-s}$,  where $b$ and $s$  are given in \eqref{b-choice} and \eqref{s-choice}, respectively. Then, 
\begin{equation}
\begin{aligned}
\int_{B_R}e^{|u|^{\frac{n}{n-2}}}dx=|B_R|\int_{0}^{\infty}e^{w^{\frac{n}{n-2}}(t)-t}dt>|B_R|\left(1+e^{\psi(\frac{n}{2})+\gamma}\right),
\end{aligned}
\end{equation}
for $n\ge 2T_0$, where $T_{0}$ is the smallest positive integer such that
$$
T_{0}\ge 1+ \frac{1+36\sigma}{17-24\gamma}+\left[1+\left(\frac{1+36\sigma}{17-24\gamma}\right)^{2}+\frac{72\sigma}{17-24\gamma}\right]^{\frac{1}{2}}\approx 51.9233.
$$ 
\end{lemma}
\begin{proof}
Firstly, we have
\begin{equation}\nonumber
\int_{0}^{\frac{n}{2}}e^{w^{\frac{n}{n-2}}(t)-t}dt=\frac{n}{2}\int_{0}^{1}e^{g(\tau)}d\tau,
\end{equation}
where
$$
g(\tau)=\left(\frac{n}{2}-1\right)\tau^{\frac{n}{n-2}}-\frac{n}{2}\tau,\;\; \tau\in [0,1].
$$
It is easy to see that $g$ is a decreasing function with $g(1)=-1$ and $g(\tau)\ge -\frac{n}{2}\tau$, for all $0<\tau\le \frac{2}{n}$. Thus
\begin{equation}\nonumber
\begin{aligned}
\int_{0}^{1}e^{g(\tau)}d\tau\ge \int_{0}^{\frac{2}{n}}e^{-\frac{n}{2}\tau}d\tau +\frac{1}{e}\left(1-\frac{2}{n}\right)=\frac{2}{n}\left(1-\frac{1}{e}\right)+\frac{1}{e}\left(1-\frac{2}{n}\right).
\end{aligned}
\end{equation}
Hence,
\begin{equation}\label{partI}
\begin{aligned}
\int_{0}^{\frac{n}{2}}e^{w^{\frac{n}{n-2}}(t)-t}dt\ge 1+\frac{1}{e}\left(\frac{n}{2}-1\right)-\frac{1}{e}.
\end{aligned}
\end{equation}
In addition,
\begin{equation}\label{partII}
\begin{aligned}
\int_{\frac{n}{2}}^{\lambda}e^{w^{\frac{n}{n-2}}(t)-t}dt=\frac{1}{e}\left(\lambda-\frac{n}{2}\right)=-\frac{1}{e}\left(\frac{n}{2}-1\right)+\left(\frac{n}{2}-1\right)e^{b-s-1}.
\end{aligned}
\end{equation}
Finally, 
\begin{equation}\label{partIII}
\begin{aligned}
\int_{\lambda}^{\infty}e^{w^{\frac{n}{n-2}}(t)-t}dt\ge e^{\lambda-1}\int_{\lambda}^{\infty}e^{-t}dt=\frac{1}{e}. 
\end{aligned}
\end{equation}
It follows from \eqref{partI},\eqref{partII} and \eqref{partIII} that
\begin{equation}
\begin{aligned}
\int_{0}^{\infty}e^{w^{\frac{n}{n-2}}(t)-t}dt\ge 1+\left(\frac{n}{2}-1\right)e^{b-s-1}.
\end{aligned}
\end{equation}
Hence, it remains to prove that
\begin{equation}\label{concentrateS}
\left(\frac{n}{2}-1\right)e^{b-s-1}> e^{\psi(\frac{n}{2})+\gamma},
\end{equation}
or equivalently
\begin{equation}\nonumber
\psi\left(\frac{n}{2}\right)+\gamma+s-b+1-\ln\left(\frac{n}{2}-1\right)<0,
\end{equation}
for $n\ge 2T_0$. We will use the auxiliary function
\begin{align*}
	\eta\left(t\right)=\psi\left(t\right)+\gamma+\left(\frac{t}{t-1}\right)^{t}\left[\left(\frac{2\sqrt{3}}{3}+1\right)\frac{1}{t-1}-1\right]+\frac{t}{t-1}+1-\ln\left(t-1\right), \; t\ge 2.
\end{align*}
Thus, taking into account \eqref{b-choice} and \eqref{s-upb}, we obtain
$$
\psi\left(\frac{n}{2}\right)+\gamma+s-b+1-\ln\left(\frac{n}{2}-1\right)<\eta\left(\frac{n}{2}\right),\;\; n\ge 4.
$$
Hence, it is now enough  to prove the following:
\begin{assertion}
Let $\sigma=1+\frac{2\sqrt{3}}{3}$. Then, the function
\begin{align*}
	\eta\left(t\right)=\psi\left(t\right)+\gamma+\left(\frac{t}{t-1}\right)^{t}\left[\frac{\sigma}{t-1}-1\right]+\frac{t}{t-1}+1-\ln\left(t-1\right),\; t\ge 2.
\end{align*}
is strictly decreasing on $[2,\infty)$ and $\eta(T_0)<0$.
\end{assertion}
\noindent Firstly, we shall prove that $\eta^{\prime}(t)<0$, for all $t\ge 2 $. From \eqref{psi}, we can write
\begin{align*}
\psi^{\prime}\left(t\right)&=\sum_{k=0}^{+\infty} \frac{1}{\left(k+t\right)^2}\le \sum_{k=0}^{+\infty} \frac{1}{\left(k+t-\frac{1}{2}\right)\left(k+t+\frac{1}{2}\right)}=\frac{1}{t-\frac{1}{2}}
\end{align*}
and consequently
\begin{equation}\nonumber
\begin{aligned}
\eta^{\prime}(t)&=\psi^{\prime}(t)-\left(\frac{t}{t-1}\right)^{t}\left\{\frac{t-(\sigma+1)}{t-1}\left[\ln\left(\frac{t}{t-1}\right)-\frac{1}{t-1}\right]+\frac{\sigma}{(t-1)^{2}}\right\}\\
& -\frac{1}{(t-1)^2}-\frac{1}{t-1}\\
&\le \left(\frac{1}{t-\frac{1}{2}}-\frac{1}{t-1}\right)-\frac{1}{(t-1)^2}\\
& -\left(\frac{t}{t-1}\right)^{t}\left\{\frac{t-(\sigma+1)}{t-1}\left[\ln\left(\frac{t}{t-1}\right)-\frac{1}{t-1}\right]+\frac{\sigma}{(t-1)^{2}}\right\}.
\end{aligned}
\end{equation}
It follows that 
\begin{equation}\nonumber
\begin{aligned}
\eta^{\prime}(t)
& <-\left(\frac{t}{t-1}\right)^{t}\left\{\frac{t-(\sigma+1)}{t-1}\left[\ln\left(\frac{t}{t-1}\right)-\frac{1}{t-1}\right]+\frac{\sigma}{(t-1)^{2}}\right\}
\end{aligned}
\end{equation}
and it is sufficient to show
$$
\frac{t-(\sigma+1)}{t-1}\left[\ln\left(\frac{t}{t-1}\right)-\frac{1}{t-1}\right]+\frac{\sigma}{(t-1)^{2}}\ge 0,\;\; t\ge 2.
$$
Setting $x=t/(t-1)$ the above inequality is equivalent to show
$$
h(x):=[x-(x-1)(\sigma+1)][\ln x-x+1]+\sigma(x-1)^2\ge 0 ,\;\; 1<x\le 2. 
$$
We have $h(x)\rightarrow 0$, as $x\rightarrow 1^{+}$. We claim that $h$ is an increasing function on $(1,2]$. Indeed, we have 
$$
h^{\prime}(x)=4\sigma x+\frac{\sigma+1}{x}-5\sigma-\sigma\ln x -1.
$$
In addition,
$$
x^{2}h^{\prime\prime}(x)=4\sigma x^2-\sigma x-(\sigma+1).
$$
Since $y(x)=4\sigma x^2-\sigma x-(\sigma+1)$, $x\in\mathbb{R}$ is  a convex parable  with minimal point at $x=\frac{1}{8}$ and $y(1)=2\sigma-1>0$, we obtain $x^{2}h^{\prime\prime}(x)>0$ for $x\in (1,2]$. Hence, $h^{\prime}>0$ on $(1,2]$ since it is an increasing function  with  $h^{\prime}(x)\rightarrow 0$ as $x\rightarrow 1^{+}$.

Next, we will prove $\eta(T_0)<0$. Firstly, since $\Gamma(x+1)=x\Gamma(x), x>0$ we have $\ln \Gamma(x+1)=\ln \Gamma(x)+\ln x$ and thus 
$$
\psi(x+1)=\psi(x)+\frac{1}{x}.
$$
Thus, since $\psi(1)+\gamma=0$ (cf. \eqref{psi}) we have 
$$
\psi(k+1)+\gamma=H_k:=\sum_{j=1}^{k}\frac{1}{k}, \;\;\mbox{for any}\;\; k\in\mathbb{N}.
$$
In particular, 
$$
\eta(k+1)=2+\gamma_{k}+\left(1+\frac{1}{k}\right)^{k+1}\left(\frac{\sigma}{k}-1\right)+\frac{1}{k},
$$
where $\gamma_k=H_k-\ln k$. In addition, from \cite[Corollary 2.13]{Qiu}  the estimate
\begin{equation}\nonumber
\gamma_{k}<\gamma+\frac{1}{2k}-\frac{\beta}{k^2},\;\;\mbox{with}\;\; \beta=\gamma-\frac{1}{2}.
\end{equation}
holds. Therefore, we can write
\begin{equation}\nonumber
\eta(k+1)<2+\gamma+\frac{1}{k}+\frac{1}{2k}+\left(1+\frac{1}{k}\right)^{k+1}\left(\frac{\sigma}{k}-1\right)-\frac{\beta}{k^2}.
\end{equation}
For $k\ge 4$, it is easy to see that
\begin{equation}\nonumber
\begin{aligned}
\left(1+\frac{1}{k}\right)^{k+1}&\geq \binom{k+1} 0+\binom{k+1} 1\frac{1}{k}+\binom{k+1}{ 2}\frac{1}{k^2}+\binom{k+1}{ 3}\frac{1}{k^3}+\binom{k+1}{4}\frac{1}{k^4}\\
&=2+\frac{17}{24}+\frac{1}{k}+\frac{1}{2k}-\frac{1}{12k}-\frac{5}{24k^2}+\frac{1}{12k^3}.\\
\end{aligned}
\end{equation}
Then, since $\left(1+\frac{1}{k}\right)^{k}\le 3$, for $k\ge 4$ and $\beta=\gamma-1/2$
\begin{equation}\nonumber
\begin{aligned}
\eta(k+1)&<2+\gamma+\frac{1}{k}+\frac{1}{2k}-\frac{\beta}{k^2}-\left(1+\frac{1}{k}\right)^{k+1}+\frac{\sigma}{k}\left(1+\frac{1}{k}\right)^{k}\left(1+\frac{1}{k}\right)\\
&<\left(\gamma-\frac{17}{24}\right)+\left(\frac{1}{12}+3\sigma\right)\frac{1}{k}+\left(\frac{17}{24}-\gamma+3\sigma\right)\frac{1}{k^2}.
\end{aligned}
\end{equation}
It follows that 
\begin{equation}\nonumber
\begin{aligned}
k^2\eta(k+1)&<\left(\gamma-\frac{17}{24}\right)\left[k^2-\left(\frac{1+36\sigma}{12}\right)\frac{24}{17-24\gamma}k-\frac{72\sigma}{17-24\gamma}-1\right]\\
&=\left(\gamma-\frac{17}{24}\right)\left[k^2-\frac{2+72\sigma}{17-24\gamma}k-\frac{72\sigma}{17-24\gamma}-1\right]\\
&=\left(\gamma-\frac{17}{24}\right)\left[\left(k-\frac{1+36\sigma}{17-24\gamma}\right)^{2}-\left(\frac{1+36\sigma}{17-24\gamma}\right)^{2}-\frac{72\sigma}{17-24\gamma}-1\right].
\end{aligned}
\end{equation}
Thus, since $\gamma<17/24$  we have $\eta(k+1)<0$ provided that 
$$
k\in\mathbb{N} \;\;\;\mbox{and}\;\; k\ge \frac{1+36\sigma}{17-24\gamma}+\left[1+\left(\frac{1+36\sigma}{17-24\gamma}\right)^{2}+\frac{72\sigma}{17-24\gamma}\right]^{\frac{1}{2}}.
$$
\end{proof}
\section{Poof of the Theorem \ref{thm1}: case $m\ge 2$}\label{generalcasethm1}
First, as well as in Section~\ref{estimate}, we  only need to  consider $u_i\in C^m (\Omega) \cap W_{\mathcal{N}}^{m,\frac{n}{m}}(\Omega)$. Second, with the help of  Proposition \ref{comparizon} and  P\'olya-Szeg\"o inequality we  iterate the same argument of the  Lemma \ref{concentratecompa}  in order to obtain our result for $m\ge 2$.  
\begin{lemma}\label{concentratecompageral}
	Let $\Omega\subset\mathbb{R}^n$ be a  smooth domain and $B_R\subset\mathbb{R}^n$ is the ball centered at the origin such that $|B_R|=|\Omega|$. Let $(u_i)\subset W_{\mathcal{N}}^{m,\frac{n}{m}}(\Omega)$ be a concentrated sequence such that $\|\nabla^m u_i\|_{\frac{n}{m}}=1$. Then there exists $(v_i)\subset W_{\mathcal{N}}^{m,\frac{n}{m}}(B_R)$ such that $u_i^*\leq v_i$. In addition,  one has
	\begin{equation}\nonumber
	\left\{\begin{aligned}
	&\|\nabla^m v_i\|_{\frac{n}{m}}= \|\nabla^m u_i\|_{\frac{n}{m}}=1\;\;\mbox{and}\;\; \lim_{i\rightarrow \infty} \int_{B_R\setminus B_r} |\Delta^\frac{m}{2} v_i|^\frac{n}{m}\ud x=0,\;\;\mbox{if}\;\;m\;\; \mbox{is even} \\
	&\|\nabla^m v_i\|_{\frac{n}{m}}\le  \|\nabla^m u_i\|_{\frac{n}{m}}=1\;\;\mbox{and}\;\; \lim_{i\rightarrow \infty} \int_{B_R\setminus B_r} |\Delta^\frac{m-1}{2} v_i|^\frac{n}{m}\ud x=0,\;\;\mbox{if}\;\;m\;\; \mbox{is odd}.
	\end{aligned}\right.
	\end{equation}
\end{lemma}

\begin{proof}
	Initially, for $m$  odd,   we will prove that  up to a subsequence
	\begin{equation}\label{concentration odd}
	\lim_{i\rightarrow \infty} \int_{\Omega\setminus B_r (x_0)} |\Delta^\frac{m-1}{2} u_i|^\frac{n}{m}\ud x=0,\qquad\mbox{for any}\quad r>0,
	\end{equation}
	where $x_0\in\overline{\Omega}$ is the concentration point of the sequence $u_i\in W_{\mathcal{N}}^{m,\frac{n}{m}}(\Omega)$.  Poincar\'e inequality yields
	\begin{equation}\label{poincineq}
	\int_{U} \left|\Delta^\frac{m-1}{2} u_i-\left(\Delta^\frac{m-1}{2} u_i\right)_U\right|^\frac{n}{m}\ud x\leq C'\int_{U} \left|\nabla\Delta^\frac{m-1}{2} u_i\right|^\frac{n}{m}\ud x,
	\end{equation}
	for $U= \Omega\setminus B_r (x_0)$ and $\left(\Delta^\frac{m-1}{2} u_i\right)_U=|U|^{-1}\int_{U}\Delta^\frac{m-1}{2} u_i \ud x$. 
	
	Since  $(\Delta^\frac{m-1}{2} u_i)\subset W_0^{1,\frac{n}{m}}(\Omega)$ is a bounded sequence, we have $\Delta^\frac{m-1}{2} u_i\rightharpoonup h$ in $W_0^{1,\frac{n}{m}}(\Omega)$  and the compact embedding gives $\Delta^\frac{m-1}{2} u_i\rightarrow h$ in $L^{\frac{n}{m}}(\Omega)$. Thus, 
	$$
	\int_{U} \left|\Delta^\frac{m-1}{2} u_i\right|^\frac{n}{m}\ud x\rightarrow \int_{U} \left|h\right|^\frac{n}{m}\ud x \quad \mbox{and}\quad 	\left(\Delta^\frac{m-1}{2} u_i\right)_U \rightarrow \left( h\right)_U.
	$$
Therefore, since $(u_i)$ is a concentrated sequence and \eqref{poincineq} holds we obtain $h=\left( h\right)_U$ a.e. in $U$. Finally, $(\Delta^\frac{m-1}{2} u_i)_{|_{\partial\Omega}}=0$ and $\Delta^\frac{m-1}{2} u_i\rightarrow h$ a.e in $U$ imply $h\equiv 0$ and thus \eqref{concentration odd}. 

In view of \eqref{concentration odd}  for both cases $m$ odd and $m$ even, we can apply 
Proposition \ref{comparizon}, the same argument used in the proof of Lemma \ref{concentratecompa} and the P\'olya-Szeg\"o inequality  to finish the proof.		
\end{proof}
To  apply  Proposition \ref{comparizon} in the proof of Lemma~\ref{concentratecompageral}, 
for either $m=2k$ or $m=2k+1$, with  $k \geq 2$ and for $f_i=\Delta^k u_i $ we can rewrite the problems \eqref{pnbcn} and \eqref{pnbc} as the following systems 
	\begin{equation}\label{double-system}
	\begin{array}{ll}
	\begin{cases}
	-\Delta u_i^1 = f_i &\mbox{in} \quad \Omega\\
	u_i^1=0 & \mbox{in}\quad \partial\Omega
	\end{cases}
	\qquad
	&
	\begin{cases}
	-\Delta v_i^1 = f_i^* &\mbox{in} \quad \Omega^*\\
	v_i^1=0 & \mbox{in}\quad \partial\Omega^*
	\end{cases}
	\\
	\quad
	\\
	\begin{cases}
	-\Delta u_i^j = u_i^{j-1} &\mbox{in} \quad \Omega\\
	u_i^j=0 & \mbox{in}\quad \partial\Omega
	\end{cases}
	\qquad
	&
	\begin{cases}
	-\Delta v_i^j = v_i^{j-1} &\mbox{in} \quad \Omega^*\\
	v_i^j=0 & \mbox{in}\quad \partial\Omega^*,
	\end{cases}
	\end{array}
	\end{equation}
	for $j=2,...,k$. Note that  $u_i^k=u_i$, $v_i^k=v_i$.  As well as in \eqref{def wi}, we denote
\begin{equation}\label{wi def2}
w_i(|x|)=v_i(x).	
\end{equation}
In \eqref{double-system},   setting
$h_i^{j-1}(\omega_{n-1}|x|^n)=v_i^{j-1}(x)$ we have
\begin{equation}\label{vijseq}
v_i^j(x)=\frac{1}{n^2\omega_{n}^{\frac{2}{n}}} \int_{\omega_{n}|x|^n}^{\omega_{n}R^n} s^{\frac{2}{n}-2}\int_{0}^{s} h_i^{j-1}(t) \ud t \ud s.
\end{equation}
With this notation, we are able to prove the following.
\begin{lemma}\label{concentrationgisup} Let  $(w_i)$ be the sequence in \eqref{wi def2}. Then,   for any $r\in(0,R)$ we have
	$$
	\int_{r}^{R}  |w'_i|^\frac{n}{m} t^{{\frac{n}{m}-1}}\ud t\rightarrow 0,\quad \mbox{as}\quad i \rightarrow \infty.
	$$
\end{lemma}
\begin{proof}
	First note that $w_i(|x|)=v_i^k(x)$, which is given iteratively by \eqref{double-system}. Thus, we  only need to iterate the  argument used in Lemma~\ref{concentrationgi}, H\"{o}lder inequality and \eqref{vijseq}  to get the result.
\end{proof}

Now, as well  as in the case $m=2$, we consider the  change of variable
\begin{equation}\label{gim}
r= R e^{-\frac{t}{n}}\quad\mbox{and} \quad g_i(t)=\beta_0(m,n)^\frac{n-m}{n} w_i(r).
\end{equation}
In order to complete the proof of Theorem~\ref{thm1}, it is sufficient to show that the sequence $(g_i)$ is under the hypotheses of  Theorem~\ref{realestimate}. To accomplish this task we will first use the result proved in \cite[Proposition 3.1]{deoliveiradoomacedo2018}. For completeness, we will prove the necessary version of this result.
\begin{proposition}\label{best constant}
	Let $p, q> 1$ and $0<R<\infty$.   Consider $n$   satisfying $n-2q>0$.
	Then, for any $u\in AC^{1}_{loc}(0,R) $ such that $\lim_{r\rightarrow R} u(r)=0$ and $\lim_{r\rightarrow R} r^{1-n}\left(r^{n-1} u\right)'=0$ we have 
	\begin{equation}\label{best inequality 1}
	\left(\int_{0}^{R}|u|^pr^{\frac{np}{q^*}-1} dr\right)^{\frac{1}{p}} \leq \frac{q^2}{(q-1)n(n-2q)} \left(\int_{0}^{R}|r^{1-n}\left(r^{n-1} u\right)'|^p r^{\frac{np}{q}-1} dr\right)^{\frac{1}{p}},
	\end{equation}
	where $	q^*=\frac{nq}{n-2q}$.
\end{proposition}
\begin{proof}
	Consider the following  change of variable
	$$
	w(t)=u(R t^{\frac{1}{2-n}}),\;\; t\ge 1.
	$$
	Note that 
	\begin{equation}\label{CLp}
	\int_{0}^{R}|u|^pr^{\frac{n p}{q^*}-1} dr=\frac{R^{\frac{n p}{q^*}}}{n-2} \int_{1}^{\infty} |w|^p t^{\frac{np}{q^*(2-n)}-1} dt=\frac{R^{\frac{np}{q^*}}}{n-2} \int_{1}^{\infty} |w|^p t^{-\frac{(n-2q)p}{q(n-2)}-1} dt
	\end{equation}
	and
	\begin{equation}\label{CDp}
	\int_{0}^{R}|r^{1-n}\left(r^{n-1} u\right)'|^p r^{\frac{np}{q}-1} dr	=(n-2)^{2p}\frac{R^{\frac{np}{q}-2p}}{n-2}  \int_{1}^{\infty}|w''(t)|^p  t^{p\frac{2q(n-1)-n}{q(n-2)}-1} dt.
	\end{equation}
	By choosing
	$$
	a=\frac{p-1}{p}\frac{2q(n-1)-n}{q(n-2)}
	$$
	we have 
	\begin{align*}
	w^p(t)=&\left(\int_{1}^{\infty} w'(z)dz \right)^p =\left(\int_{1}^{t}  \int_z^{\infty} -w''(s) ds dz \right)^p=\left(\int_{1}^{t}  \int_z^{\infty} -w''(s) \frac{s^{a}}{s^a} ds dz \right)^p\\
	&\leq  \left( \left[\int_{1}^{t}  \int_z^{\infty} |w''(s)|^p s^{ap} ds dz\right]^{1/p}  \left[\int_{1}^{t}  \int_z^{\infty} s^{-ap'} ds dz\right]^{1/p'}\right)^p\\
	&\leq   \left[\int_{1}^{t}  \int_z^{\infty} |w''(s)|^p s^{ap} ds dz\right]  \left[  \frac{q^2(n-2)^2}{n(q-1)(n-2q)}\left(t^{\frac{n-2q}{q(n-2)}} -1 \right)\right]^{p/p'}.\\
	\end{align*}	
	Then, we have 
	\begin{align*}
	\int_{1}^{\infty} |w|^p  t^{-\frac{(n-2q)p}{q(n-2)}-1} dt &\leq  \int_{1}^{\infty} \left[\int_{1}^{t}  \int_z^{\infty} |w''(s)|^p s^{ap} ds dz\right] \\
	&\qquad\qquad\times\left[   \frac{q^2(n-2)^2}{n(q-1)(n-2q)} \left(t^{\frac{n-2q}{q(n-2)}} -1 \right)\right]^{p-1}  t^{-\frac{(n-2q)p}{q(n-2)}-1}dt\\ 
	&= \left(\frac{q^2(n-2)^2}{n(q-1)(n-2q)}\right)^{p-1}   \int_{1}^{\infty} \int_{1}^{s}  |w''(s)|^p s^{ap}     \int_z^{\infty} t^{-\frac{n-2q}{q(n-2)}-1} dt dz ds\\
	&\leq \frac{\left(\frac{q^2(n-2)^2}{n(q-1)(n-2q)}\right)^{p-1} }{ \frac{n-2q}{q(n-2)}\left(-\frac{n-2q}{q(n-2)}+1\right)}  \int_{1}^{\infty} |w''(s)|^p \left(s^{-\frac{n-2q}{q(n-2)}+1}-1\right)  s^{(1-p)\frac{n-2q(n-1)}{q(n-2)} } ds\\
	&\leq \left(\frac{q^2(n-2)^2}{n(q-1)(n-2q)}\right)^{p}   \int_{1}^{\infty} |w''(s)|^p s^{p\frac{2q(n-1)-n}{q(n-2)}-1}ds.
	\end{align*}	
	Combining this last inequality with \eqref{CLp} and \eqref{CDp} we get \eqref{best inequality 1}.
\end{proof}
The next result represents the version of Lemma~\ref{bestconstlemma} for $m\ge 2$.

\begin{lemma}\label{constant 2to1} For any $ u \in AC_{\mathrm{L}}(0,R)$, we have
	$$
	\int_{0}^{R}  |u|^\frac{n}{m} r^{\frac{n}{m}\left( 1- n\right)+\frac{n}{m}-1}\ud r\leq \frac{1}{(n-2)^{\frac{n}{m}}} \int_{0}^{R}  |u'|^\frac{n}{m} r^{\frac{n}{m}\left( 1- n\right)+\frac{2n}{m}-1}\ud r.
	$$
	In particular, for  $w_i$ defined as in \eqref{wi def2} 
	\begin{equation}\label{bestconstant2}
		(n-2)^{\frac{n}{m}} \omega_{n-1}\int_{0}^{R}  |w'_i|^\frac{n}{m} r^{\frac{n}{m}-1}\ud r\leq  \omega_{n-1}\int_{0}^{R}  |r^{1-n}\left(r^{n-1} w'_i\right)'|^\frac{n}{m} r^{\frac{2n}{m}-1}\ud r
	\end{equation}
	holds.
\end{lemma}	
\begin{proof}
It is a consequence of the Corollary~\ref{PCoro1}, item $(i)$ with the choice
	\begin{equation}\label{our-choice2}
	p=q=\frac{n}{m},\quad \alpha=\frac{n}{m}\left( 1- n\right)+\frac{2n}{m}-1 
	\quad\mbox{and}\quad 	\theta=\frac{n}{m}\left( 1- n\right)+\frac{n}{m}-1.
	\end{equation}
	Indeed, we can check that this choice satisfies $ \alpha-p+1=\frac{n}{m}(2-n)<0$ and $p=\alpha-\theta$. Hence, we must have 
	\begin{equation}\nonumber
	(p-1)^{1-\frac{1}{p}}\frac{1}{p-1-\alpha}\le \mathcal{C}_{L}\le \frac{p}{p-1-\alpha}=\frac{1}{n-2}
	\end{equation}	
	which completes the proof.
\end{proof}
\begin{rem} \label{remakend} Analogous to the identity \eqref{bcm=2},  since
 \begin{equation}\nonumber
\omega_{n-1}=\frac{2\pi^{\frac{n}{2}}}{\Gamma\left(\frac{n}{2}\right)}\quad\mbox{and}\quad  \Gamma(x+p)=\Gamma(x)\prod_{j=0}^{p-1}(x+j),\;\;p\in\mathbb{N}
 \end{equation}
we can write the following expressions for the critical exponent $\beta_0(m,n)$: If $m=2k$ is even
	\begin{align*}
		\frac{n}{\omega_{n-1}}\left[ \frac{\pi^{\frac{n}{2}}2^m\Gamma\left(\frac{m}{2}\right)}{\Gamma\left(\frac{n-m}{2}\right)}\right]^{{n}/{(n-m)}}	
		&= \left[  n^{\frac{n-m}{n}}{\omega_{n-1}}^{\frac{m}{n}}(n-2){\prod_{j=0}^{k-2} (n-m+2j)(m-2j-2)}\right]^{{n}/{(n-m)}}\\
	\end{align*}
and, for $m=2k+1$ odd, we have
\begin{align*}
	\frac{n}{\omega_{n-1}}\left[ \frac{\pi^{\frac{n}{2}}2^m\Gamma\left(\frac{m+1}{2}\right)}{\Gamma\left(\frac{n-m+1}{2}\right)}\right]^{{n}/{(n-m)}}	
	&= \left[  n^{\frac{n-m}{n}}{\omega_{n-1}}^{\frac{m}{n}}{\prod_{j=0}^{k-1} (n-m+2j+1)(m-2j-3)}\right]^{{n}/{(n-m)}}.
\end{align*}
\end{rem}
Finally, we will prove that the sequence $(g_i)$ in \eqref{gim} satisfies the conditions of Carleson-Chang type estimate Theorem~\ref{realestimate}.
\begin{proposition}\label{reduction} Let
 $(g_i)$ be the sequence given in \eqref{gim}. Then	
$$
\lim_{i\rightarrow\infty}\int_{0}^{A}|g^{\prime}_i|^{\frac{n}{m}}\ud t=0\quad \mbox{and}\quad \int_{0}^{+\infty}  | g_i'|^\frac{n}{m} \ud t\leq 1 
$$
for any $A>0$.
\end{proposition}
\begin{proof}
Initially, directly from Lemma~\ref{concentrationgisup} we get 
	$$\lim_{i\rightarrow\infty}\int_{0}^{A}|g^{\prime}_i|^{\frac{n}{m}}\ud t\rightarrow 0,\quad\forall\; A>0.$$
In order to prove the inequality we will divide the estimation in two cases:
\paragraph{Even case: $m=2k$.} 
Proposition~\ref{best constant}  with the choice $p=\frac{n}{m}=\frac{n}{2k}$,  $q_{j+1}=q_j^*=\frac{nq_j}{n-2q_{j}}$ with $q=p=q_0=\frac{n}{m}=\frac{n}{2k}$ and
	$
	u=\Delta^j w_i
	$
	 on the interval $(0,R)$,  yields
	\begin{equation}\label{iteration}
		\left(\int_{0}^{R}|\Delta^{k-j-1} w_i|^pr^{\frac{np}{q_{j+1}}-1} dr\right)^{\frac{1}{p}} \leq \frac{q_{j}^2}{(q_{j}-1)n(n-2q_{j})} \left(\int_{0}^{R}|\Delta^{k-j} w_i|^p r^{\frac{np}{q_{j}}-1} dr\right)^{\frac{1}{p}}.
	\end{equation}
Here we are denoting $\Delta u=r^{1-n}\left(r^{n-1} u'\right)'$ and $\Delta^j u=\Delta \Delta^{j-1} u$.	By iterating \eqref{iteration} we get
	\begin{align*}
	\left(\int_{0}^{R}|\Delta w_i|^pr^{\frac{2n}{m}-1} dr\right)^{\frac{1}{p}} &=	\left(\int_{0}^{R}|\Delta w_i|^pr^{\frac{np}{q_{k-1}}-1} dr\right)^{\frac{1}{p}} \\
&\leq \prod_{j=0}^{k-2}\frac{q_{j}^2}{(q_{j}-1)n(n-2q_{j})} \left(\int_{0}^{R}|\Delta^k w_i|^p r^{n-1} dr\right)^{\frac{1}{p}}.
	\end{align*}
Since 
$$
q_{j+1}=\frac{np}{n-2jp}
$$ we can also write 
\begin{equation}\label{intconstant}
		\left(\int_{0}^{R}\left|\Delta w_i\right|^\frac{n}{m} r^{\frac{2n}{m}-1} dr\right)^{\frac{m}{n}}\leq \prod_{j=0}^{k-2}\frac{1}{(n-m+2j)(m-2j-2)}\left(\int_{0}^{R}\left|\Delta^k w_i\right|^\frac{n}{m} r^{n-1} dr\right)^{\frac{m}{n}}.
\end{equation}
In view of Remark~\ref{remakend}, we can write
\begin{equation}\label{gidef2}
r= R e^{-\frac{t}{n}}\quad\mbox{and}\quad	g_i(t)=\left(n^{\frac{n-m}{n}}{\omega_{n-1}}^{\frac{m}{n}}(n-2){\prod_{j=0}^{k-2} (n-m+2j)(m-2j-2)}\right) w_i(r).
\end{equation}
Then \eqref{intconstant}, \eqref{bestconstant2} and \eqref{gidef2} yield
$$
	\int_{0}^{\infty}  |g'_i|^\frac{n}{m} \ud t={\omega_{n-1}}\left((n-2){\prod_{j=0}^{k-2} (n-m+2j)(m-2j-2)}\right)^\frac{n}{m} \int_{0}^{R} \left| w'_i(r)\right|^\frac{n}{m} r^{\frac{n}{m}-1}\ud r\leq 1.
$$

\paragraph{Odd case: $m=2k+1$.}  First we will prove that 
\begin{equation}\label{imparpar}
	\left(\int_{0}^{R}\left|\Delta^k w_i\right|^\frac{n}{m} r^{n\frac{m-1}{m}-1} dr\right)^{\frac{m}{n}}\leq \frac{1}{m-1}\left(\int_{0}^{R}\left|\left(\Delta^k w_i\right)'\right|^\frac{n}{m} r^{n-1} dr\right)^{\frac{m}{n}}.
\end{equation}
We will apply the Corollary~\ref{PCoro1}, item $(ii)$.  Firstly, note that $u=\Delta^k w_i\in AC_{\mathrm{R}} (0, R)$.  Further, the choice
$$
p=q=\frac{n}{m}, \quad \alpha=n-1\quad\mbox{and}\quad \theta=n\frac{m-1}{m}-1
$$
implies $ \alpha-p+1=n\left(1-{1}/{m}\right)>0$ and $p=\alpha-\theta$. Then, the best possible constant $\mathcal{C}_{R}$ must satisfy
		\begin{equation}\nonumber
		(p-1)^{\frac{p-1}{p}}\frac{1}{\alpha-p-1}\leq \mathcal{C}_R\leq  \frac{p}{\alpha-p-1}=\frac{1}{m-1}.
		\end{equation}
This proves \eqref{imparpar}.  Now, similarly to the even case, by iterating the Proposition \ref{best constant} with $p=\frac{n}{m}=\frac{n}{2k}$,  $q_{j+1}=q_j^*=\frac{nq_j}{n-2q_{j}}$ with $q=q_0=\frac{n}{m-1}=\frac{n}{2k}$, we obtain
 \begin{equation}\label{intconstant2}
 	\left(\int_{0}^{R}\left|\Delta w_i\right|^\frac{n}{m} r^{\frac{2n}{m}-1} dr\right)^{\frac{m}{n}}\leq \prod_{j=0}^{k-2}\frac{1}{(n-m+2j+1)(m-2j-3)}\left(\int_{0}^{R}\left|\Delta^k w_i\right|^\frac{n}{m} r^{n\frac{m-1}{m}-1} dr\right)^{\frac{m}{n}}.
 \end{equation}
Then, by Lemma \ref{constant 2to1}, inequality \eqref{imparpar} and \eqref{intconstant2}, if we take (cf. Remark~\ref{remakend})
$$
	g_i(t)=\left(n^{\frac{n-m}{n}}{\omega_{n-1}}^{\frac{m}{n}}{\prod_{j=0}^{k-1} (n-m+2j+1)(m-2j-3)}\right) w_i(r),
$$
we  have
$$
\int_{0}^{\infty}  |g'_i|^\frac{n}{m} \ud t={\omega_{n-1}}\left({\prod_{j=0}^{k-1} (n-m+2j+1)(m-2j-3)}\right)^\frac{n}{m} \int_{0}^{R} \left| w'_i(r)\right|^\frac{n}{m} r^{\frac{n}{m}-1}\ud r\leq 1.
$$
\end{proof}
\begin{proof}[Proof of  Theorem \ref{thm1}]
	At this point,  from  Proposition~\ref{reduction},  the result follows from  Theorem~\ref{realestimate}.
\end{proof}
\bibliographystyle{siam}

\begin{thebibliography}{10}
	
	\bibitem{Adams1988}
	{\sc D.~R. Adams}, {\em A sharp inequality of {J}. {M}oser for higher order
		derivatives}, Ann. of Math. (2), 128 (1988), pp.~385--398.
	
	\bibitem{Alberico08}
	{\sc A.~Alberico}, {\em Moser type inequalities for higher-order derivatives in
		{L}orentz spaces}, Potential Anal., 28 (2008), pp.~389--400.
	
	\bibitem{Alvino96}
	{\sc A.~Alvino, V.~Ferone, and G.~Trombetti}, {\em Moser-type inequalities in
		{L}orentz spaces}, Potential Anal., 5 (1996), pp.~273--299.
	
	\bibitem{s-functions}
	{\sc G.~E. Andrews, R.~Askey, and R.~Roy}, {\em Special functions}, vol.~71 of
	Encyclopedia of Mathematics and its Applications, Cambridge University Press,
	Cambridge, 1999.
	
	\bibitem{FardounRachid2006}
	{\sc P.~Baird, A.~Fardoun, and R.~Regbaoui}, {\em {$Q$}-curvature flow on
		4-manifolds}, Calc. Var. Partial Differential Equations, 27 (2006),
	pp.~75--104.
	
	\bibitem{FardounRachid2009}
	\leavevmode\vrule height 2pt depth -1.6pt width 23pt, {\em Prescribed
		{$Q$}-curvature on manifolds of even dimension}, J. Geom. Phys., 59 (2009),
	pp.~221--233.
	
	\bibitem{Carleson-Chang}
	{\sc L.~Carleson and S.-Y.~A. Chang}, {\em On the existence of an extremal
		function for an inequality of {J}. {M}oser}, Bull. Sci. Math. (2), 110
	(1986), pp.~113--127.
	
	\bibitem{MR3155968}
	{\sc D.~Cassani, B.~Ruf, and C.~Tarsi}, {\em A {M}oser type inequality in
		{Z}ygmund spaces without boundary conditions}, in Recent trends in nonlinear
	partial differential equations. {II}. {S}tationary problems, vol.~595 of
	Contemp. Math., Amer. Math. Soc., Providence, RI, 2013, pp.~121--133.
	
	\bibitem{luluzhu20}
	{\sc L.~Chen, G.~Lu, and M.~Zhu}, {\em Existence and nonexistence of extremals
		for critical {A}dams inequalities in {$\Bbb R^4$} and {T}rudinger-{M}oser
		inequalities in {$\Bbb R^2$}}, Adv. Math., 368 (2020), pp.~107143, 61.
	
	\bibitem{Cianchi08}
	{\sc A.~Cianchi}, {\em Moser-{T}rudinger trace inequalities}, Adv. Math., 217
	(2008), pp.~2005--2044.
	
	\bibitem{doOdeOliveira2014}
	{\sc J.~F. de~Oliveira and J.~M. do~\'O}, {\em Trudinger-{M}oser type
		inequalities for weighted {S}obolev spaces involving fractional dimensions},
	Proc. Amer. Math. Soc., 142 (2014), pp.~2813--2828.
	
	\bibitem{DelaTorre}
	{\sc A.~DelaTorre and G.~Mancini}, {\em Improved {A}dams-type inequalities and
		their extremals in dimension $2m$}, Commun. Contemp. Math.,  (2020).
	
	\bibitem{MR4289908}
	\leavevmode\vrule height 2pt depth -1.6pt width 23pt, {\em Improved
		{A}dams-type inequalities and their extremals in dimension {$2m$}}, Commun.
	Contemp. Math., 23 (2021), pp.~Paper No. 2050043, 52.
	
	\bibitem{doOMacedo2014}
	{\sc J.~M. do~\'O and A.~C. Macedo}, {\em Concentration-compactness principle
		for an inequality by {D}. {A}dams}, Calc. Var. Partial Differential
	Equations, 51 (2014), pp.~195--215.
	
	\bibitem{MacedodoO2015}
	\leavevmode\vrule height 2pt depth -1.6pt width 23pt, {\em Adams type
		inequality and application for a class of polyharmonic equations with
		critical growth}, Adv. Nonlinear Stud., 15 (2015), pp.~867--888.
	
	\bibitem{deoliveiradoomacedo2018}
	{\sc J.~M. do~\'{O}, A.~C. Macedo, and J.~F. de~Oliveira}, {\em A sharp
		{A}dams-type inequality for weighted {S}obolev spaces}, Q. J. Math., 71
	(2020), pp.~517--538.
	
	\bibitem{MR2721666}
	{\sc J.~M. do~\'O and Y.~Yang}, {\em A quasi-linear elliptic equation with
		critical growth on compact {R}iemannian manifold without boundary}, Ann.
	Global Anal. Geom., 38 (2010), pp.~317--334.
	
	\bibitem{FardounRachid2012}
	{\sc A.~Fardoun and R.~Regbaoui}, {\em {$Q$}-curvature flow for {GJMS}
		operators with non-trivial kernel}, J. Geom. Phys., 62 (2012),
	pp.~2321--2328.
	
	\bibitem{MR1171306}
	{\sc M.~Flucher}, {\em Extremal functions for the {T}rudinger-{M}oser
		inequality in {$2$} dimensions}, Comment. Math. Helv., 67 (1992),
	pp.~471--497.
	
	\bibitem{Fontana12}
	{\sc L.~Fontana and C.~Morpurgo}, {\em Sharp {M}oser-{T}rudinger inequalities
		for the {L}aplacian without boundary conditions}, J. Funct. Anal., 262
	(2012), pp.~2231--2271.
	
	\bibitem{Fontana20}
	\leavevmode\vrule height 2pt depth -1.6pt width 23pt, {\em Adams inequalities
		for {R}iesz subcritical potentials}, Nonlinear Anal., 192 (2020), pp.~111662,
	32.
	
	\bibitem{GazzolaGrunauSweers2010}
	{\sc F.~Gazzola, H.-C. Grunau, and G.~Sweers}, {\em Optimal {S}obolev and
		{H}ardy-{R}ellich constants under {N}avier boundary conditions}, Ann. Mat.
	Pura Appl. (4), 189 (2010), pp.~475--486.
	
	\bibitem{HudsonLeck}
	{\sc S.~Hudson and M.~Leckband}, {\em Extremals for a {M}oser-{J}odeit
		exponential inequality}, Pacific J. Math., 206 (2002), pp.~113--128.
	
	\bibitem{YUDO61}
	{\sc V.~I. Judovi\v{c}}, {\em Some estimates connected with integral operators
		and with solutions of elliptic equations}, Dokl. Akad. Nauk SSSR, 138 (1961),
	pp.~805--808.
	
	\bibitem{Lamlu12hei}
	{\sc N.~Lam and G.~Lu}, {\em Sharp {M}oser-{T}rudinger inequality on the
		{H}eisenberg group at the critical case and applications}, Adv. Math., 231
	(2012), pp.~3259--3287.
	
	\bibitem{LiRuf2008}
	{\sc Y.~Li and B.~Ruf}, {\em A sharp {T}rudinger-{M}oser type inequality for
		unbounded domains in {$\Bbb R^n$}}, Indiana Univ. Math. J., 57 (2008),
	pp.~451--480.
	
	\bibitem{MR1333394}
	{\sc K.-C. Lin}, {\em Extremal functions for {M}oser's inequality}, Trans.
	Amer. Math. Soc., 348 (1996), pp.~2663--2671.
	
	\bibitem{Lions85}
	{\sc P.-L. Lions}, {\em The concentration-compactness principle in the calculus
		of variations. {T}he limit case. {I}}, Rev. Mat. Iberoamericana, 1 (1985),
	pp.~145--201.
	
	\bibitem{MR2483717}
	{\sc G.~Lu and Y.~Yang}, {\em Adams' inequalities for bi-{L}aplacian and
		extremal functions in dimension four}, Adv. Math., 220 (2009),
	pp.~1135--1170.
	
	\bibitem{Moser1970/71}
	{\sc J.~Moser}, {\em A sharp form of an inequality by {N}. {T}rudinger},
	Indiana Univ. Math. J., 20 (1970/71), pp.~1077--1092.
	
	\bibitem{NolascoTarantello1998}
	{\sc M.~Nolasco and G.~Tarantello}, {\em On a sharp {S}obolev-type inequality
		on two-dimensional compact manifolds}, Arch. Ration. Mech. Anal., 145 (1998),
	pp.~161--195.
	
	\bibitem{OpicKufner1990}
	{\sc B.~Opic and A.~Kufner}, {\em Hardy-type inequalities}, vol.~219 of Pitman
	Research Notes in Mathematics Series, Longman Scientific \& Technical,
	Harlow, 1990.
	
	\bibitem{P0H065}
	{\sc S.~I. Pohozaev}, {\em The sobolev embedding in the case $pl = n$},
	Proceedings of the Technical Scientific Conference on Advances of Scientific
	Research, 138 (1965), pp.~158--170.
	
	\bibitem{Qiu}
	{\sc S.-L. Qiu and D.~M. Vuorinen}, {\em Some properties of the gamma and psi
		functions, with applications}, Mathematics of computation, 74 (2004),
	pp.~723--742.
	
	\bibitem{Ruf2005}
	{\sc B.~Ruf}, {\em A sharp {T}rudinger-{M}oser type inequality for unbounded
		domains in {$\Bbb R^2$}}, J. Funct. Anal., 219 (2005), pp.~340--367.
	
	\bibitem{RufSani2013}
	{\sc B.~Ruf and F.~Sani}, {\em Sharp {A}dams-type inequalities in
		{$\Bbb{R}^n$}}, Trans. Amer. Math. Soc., 365 (2013), pp.~645--670.
	
	\bibitem{Tarek2017}
	{\sc T.~Saanouni}, {\em Fourth-order damped wave equation with exponential
		growth nonlinearity}, Ann. Henri Poincar\'e, 18 (2017), pp.~345--374.
	
	\bibitem{Federica2013}
	{\sc F.~Sani}, {\em A biharmonic equation in {$\Bbb R^4$} involving
		nonlinearities with critical exponential growth}, Commun. Pure Appl. Anal.,
	12 (2013), pp.~405--428.
	
	\bibitem{MR970849}
	{\sc M.~Struwe}, {\em Critical points of embeddings of {$H^{1,n}_0$} into
		{O}rlicz spaces}, Ann. Inst. H. Poincar\'{e} Anal. Non Lin\'{e}aire, 5
	(1988), pp.~425--464.
	
	\bibitem{Talenti1976}
	{\sc G.~Talenti}, {\em Elliptic equations and rearrangements}, Ann. Scuola
	Norm. Sup. Pisa Cl. Sci. (4), 3 (1976), pp.~697--718.
	
	\bibitem{Tarsi2012}
	{\sc C.~Tarsi}, {\em Adams' inequality and limiting {S}obolev embeddings into
		{Z}ygmund spaces}, Potential Anal., 37 (2012), pp.~353--385.
	
	\bibitem{Trudinger67}
	{\sc N.~S. Trudinger}, {\em On imbeddings into {O}rlicz spaces and some
		applications}, J. Math. Mech., 17 (1967), pp.~473--483.
	
	\bibitem{CernyCianchiHencl13}
	{\sc R.~\v{C}ern\'{y}, A.~Cianchi, and S.~Hencl}, {\em
		Concentration-compactness principles for {M}oser-{T}rudinger inequalities:
		new results and proofs}, Ann. Mat. Pura Appl. (4), 192 (2013), pp.~225--243.
	
	\bibitem{YangSuKong2016}
	{\sc Q.~Yang, D.~Su, and Y.~Kong}, {\em Sharp {M}oser-{T}rudinger inequalities
		on {R}iemannian manifolds with negative curvature}, Ann. Mat. Pura Appl. (4),
	195 (2016), pp.~459--471.
	
	\bibitem{Yang2012}
	{\sc Y.~Yang}, {\em Trudinger-{M}oser inequalities on complete noncompact
		{R}iemannian manifolds}, J. Funct. Anal., 263 (2012), pp.~1894--1938.
	
	\bibitem{MR4273149}
	{\sc M.~Zhu and L.~Wang}, {\em Adams' inequality with logarithmic weights in
		{$\Bbb{R}^4$}}, Proc. Amer. Math. Soc., 149 (2021), pp.~3463--3472.
	
\end{thebibliography}

\end{document}